\documentclass{article}%
\usepackage[T1]{fontenc}
\usepackage[francais,english]{babel}
\usepackage{amsthm}
\usepackage{amsmath}
\usepackage{amssymb}
\usepackage{mathrsfs}
\usepackage{stmaryrd}
\usepackage{thmtools}
\usepackage{graphicx}
\usepackage[top=2.5cm, bottom=2.5cm, left=2cm, right=2cm, headheight=0.5cm,
footskip=60pt]{geometry}
\usepackage{fancyhdr}
\usepackage{dsfont}
\usepackage{setspace}
\usepackage{listings}
\usepackage{amsfonts}%
\providecommand{\U}[1]{\protect\rule{.1in}{.1in}}
\fancypagestyle{plain}{
	\fancyhead{}
	
}

\newtheorem{theorem}{Theorem}

\newtheorem{corollary}[theorem]{Corollary}

\declaretheoremstyle[bodyfont=\normalfont]{body}
\declaretheorem[name=Remark , style=body]{remark}
\declaretheorem[name=Lemma , style=body]{lemma}
\declaretheorem[name=Proposition , style=body]{proposition}
\declaretheorem[name=Definition , style=body]{definition}

\numberwithin{lemma}{section}
\numberwithin{proposition}{section}
\numberwithin{equation}{section}
\numberwithin{remark}{section}
\numberwithin{definition}{section}
\numberwithin{theorem}{section}
\allowdisplaybreaks
\begin{document}

\title{Long time existence results for bore-type initial data for BBM-Boussinesq systems}
\maketitle

\begin{abstract}
In this paper we deal with the long time existence for the Cauchy problem
associated to BBM-type Boussinesq systems of equations which are asymptotic
models for long wave, small amplitude gravity surface water waves. As opposed
to previous papers devoted to the long time existence issue, we consider
initial data with nontrivial behaviour at infinity which may be used to model
bore propagation.

\begin{description}
\item[Keywords] Boussinesq systems, long time existence, bore propagation;

\end{description}
\end{abstract}
\tableofcontents
\section{Introduction}

\subsection{The $abcd$ Boussinesq systems}

The following systems\ of PDEs were introduced in \cite{Bona1} as asymptotic
models for studying long wave, small amplitude gravity surface water waves:%

\begin{equation}
\left\{
\begin{array}
[c]{l}%
\left(  I-\varepsilon b\Delta\right)  \partial_{t}\eta+\operatorname{div}%
V+a\varepsilon\operatorname{div}\Delta V+\varepsilon\operatorname{div}\left(
\eta V\right)  =0,\\
\left(  I-\varepsilon d\Delta\right)  \partial_{t}V+\nabla\eta+c\varepsilon
\nabla\Delta\eta+\varepsilon V\cdot\nabla V=0.
\end{array}
\right.  \label{eq1}%
\end{equation}
In system \eqref{eq1}, $\varepsilon$ is a small parameter while for all
$\left(  t,x\right)  \in\left[  0,\infty\right)  \times\mathbb{R}^{n}$:
\[
\left\{
\begin{array}
[c]{c}%
\eta=\eta\left(  t,x\right)  \in\mathbb{R},\\
V=V\left(  t,x\right)  \in\mathbb{R}^{n}.
\end{array}
\right.
\]
The variable $\eta$ is an approximation of the deviation of the free surface
of the water from the rest state while $V$ is an approximation of the fluid
velocity. The above family of systems is derived from the classical
mathematical formulation of the water waves problem. In many applications the
water waves problem raises a significant number of difficulties from both a
theoretical and a numerical point of view. This is the reason why approximate
models have been established, each of them dealing with some particular
physical regimes.

The $abcd$ systems of equations deal with the so called Boussinesq regime
which is explained now. Consider a layer of incompressible, irrotational,
perfect fluid flowing through a canal with flat bottom represented by the
plane:%
\[
\{\left(  x,y,z\right)  :z=-h\},
\]
where $h>0$ and assume that the free surface resulting from an initial
perturbation of the steady state can be described as being the graph of a
function $\eta$ over the flat bottom. Also, consider the following quantities:
$A=\max_{x,y,t}\left\vert \eta\right\vert $ the maximum amplitude encountered
in the wave motion and $l$ the smallest wavelength for which the flow has
significant energy. Then, the Boussinesq regime is characterized by the
following parameters:%
\begin{equation}
\alpha=\frac{A}{h},\text{ }\beta=\left(  \frac{h}{l}\right)  ^{2},\text{
}S=\frac{\alpha}{\beta}, \label{Bous}%
\end{equation}
which are supposed to obey the following relations:%
\[
\alpha\ll1,\text{ }\beta\ll1\text{ and }S\approx1.
\]
Supposing for simplicity that $S=1$ and choosing $\varepsilon=\alpha=\beta$,
the systems \eqref{eq1} are derived back in \cite{Bona1} by a formal series
expansion and by neglecting the second and higher order terms. Actually, the
zeros on the right hand side of \eqref{eq1} can be viewed as the $O\left(
\varepsilon^{2}\right)  $ terms neglected in establishing \eqref{eq1}. The
parameters $a,b,c,d$ are also restricted by the following relation:%
\begin{equation}
a+b+c+d=\frac{1}{3}. \label{suma}%
\end{equation}
\ Asymptotic models taking into account general topographies of the bottom
were also derived, see \cite{Chaz}. In this situation, one has to furthermore
distinguish between two different regimes: small respectively strong
topography variations. In \cite{Chen} time variating bottoms are considered. A
systematic study of asymptotic models for the water waves problem along with
their rigorous justification can be found for instance in \cite{Lannes}. Let
us also point out that the only values of $n$ for which \eqref{eq1} is
physically relevant are $n=1,2$.

The study of the local well-posedness of the $abcd$ systems is the subject
under investigation in several papers beginning with \cite{Bona1} where,
besides deriving the family of systems \eqref{eq1}, it is shown that the
linearized equation near the null solution of \eqref{eq1} is well posed in two
generic cases, namely:%
\begin{align}
a  &  \leq0,\text{ }c\leq0,\text{ }b\geq0,\text{ }d\geq0\label{1111}\\
\text{or }a  &  =c\geq0\text{ and }b\geq0,\text{ }d\geq0. \label{2222}%
\end{align}
In \cite{Bona2}, the sequel of \cite{Bona1}, attention is given to the
well-posedness of the nonlinear systems and of some other higher order
(formally, more accurate) Boussinesq systems. Other papers that are dealing
with the well-posedness of the $abcd$ systems, for some cases that are not
treated in \cite{Bona2}, are \cite{Anh}, \cite{Bona3}, \cite{Saut2}.

The rigorous justification of the fact that systems \eqref{eq1} do approximate
the water waves problem has been carried on in \cite{Bona4} (see also
\cite{Lannes}). In this paper, the authors prove that the error estimate
between the solution of \eqref{eq1} and the water waves system at time $t$ is
of order $O\left(  \varepsilon^{2}t\right)  $. It is for this reason that on
time scales larger than $O\left(  \varepsilon^{-2}\right)  $ the solutions of
\eqref{eq1} stop being relevant approximations of the original problem.

The above error estimate result, leads to consider the so-called long time
existence\footnote{abbreviated l.t.e. in the following} problem which we will
explain now. First of all, global existence theory of solutions of \eqref{eq1}
is for the moment not at reach. Indeed, the only global results known are
available in dimension $1$ for the case:
\[
a=b=c=0,\text{ }d>0,
\]
which was studied by Amick in \cite{Amick} and Schonbek in \cite{Schon} and in
the more general case%
\[
b=d>0,\text{ }a\leq0,~c<0,
\]
which is somehow not satisfactory because one has to assume some smallness
condition on the initial data, see \cite{Bona2}. Second of all, as we
mentioned above, systems \eqref{eq1} are reliable from a practical point of
view only on time scales smaller than $O\left(  \varepsilon^{-2}\right)  $
and, as it turns out, on time intervals of order $O\left(  \varepsilon
^{-1}\right)  $ the error estimate remains of small order i.e. $O\left(
\varepsilon\right)  $. We also mention that on time scales of order $O\left(
\varepsilon^{-1/2}\right)  $, systems \eqref{eq1} behave like the linear wave
equation thus, existence results where the solution lives on such time scales
are not susceptible of having any predicting features. Also, at time scales of
order $O\left(  \varepsilon^{-1}\right)  $ the dispersive and nonlinear
effects will have an order one contribution to the wave's evolution.

All the above considerations have led people to consider the l.t.e. problem
which consists in constructing solutions of \eqref{eq1} for which the maximal
time of existence is of formal order $O\left(  \varepsilon^{-1}\right)  $.
Except for the one-dimensional previously mentioned cases, the first long time
existence result was obtained in \cite{Saut1} for the case
\[
a,c<0,\text{ }b,d>0
\]
and for the so called BBM-BBM case corresponding to%
\begin{equation}
a=c=0,\text{ }b,d>0. \label{BBMcase}%
\end{equation}
The proof adopted in this paper relies on the Nash--Moser theorem and involves
a loss of derivatives as well as a relatively high level of regularity. The
l.t.e. problem received another satisfactory answer in \cite{Saut3}, see also
\cite{Saut4}, where the case \eqref{1111} was treated and long time existence
for the Cauchy problem was systematically proved, provided that the initial
data lies in some Sobolev spaces. In \cite{Burtea1}, we used a different
method, from the one applied in \cite{Saut3} and \cite{Saut4}, which is based
on an energy method applied for spectrally localized equations, in order to
obtain l.t.e. results for most of the parameters verifying \eqref{1111}. In
particular we managed to lower the regularity assumptions needed in order to
develop the l.t.e theory. When considering small variations of the bottom
topography the methods presented in \cite{Saut4}, \cite{Saut3} and
\cite{Burtea1} adapt without to much extra effort. Regarding the strongly
variating bottoms, l.t.e. results can be found in \cite{Mes1}.

The difficulty in obtaining such results comes from the lack of symmetry of
\eqref{eq1} owing to the $\varepsilon\eta\operatorname{div}V$ term from the
first equation of the $abcd$-system . Because of the dispersive operators
$-\varepsilon b\Delta\partial_{t}+a\operatorname{div}\Delta$, $-\varepsilon
d\Delta\partial_{t}+c\nabla\Delta$, classical symmetrizing techniques for
hyperbolic systems of PDE's are very hard to apply.

The starting point of the present work is the paper \cite{Bona3} where results
pertaining to bore propagation are established for the BBM-BBM case, $\left(
\text{\ref{BBMcase}}\right)  $. The l.t.e. problem has been studied for
initial data belonging to Sobolev function spaces $H^{s}$ with the index of
regularity satisfying (at least):
\[
s>\frac{n}{2}+1.
\]
This corresponds to initial perturbation of the rest state that are
essentially localized in the space variables. Indeed, when $s$ is chosen as
above, $H^{s}$ is embedded in $\mathcal{C}_{0}$ the class of continuous
bounded functions which vanish at infinity. Because bore-type data manifest
non trivial behavior at infinity, one must of course change the functional
setting of the initial data. It is exactly this issue that we address in the
following: the l.t.e. problem for initial data which can be used to
successfully model bore-propagation. As far as we know, these are the first
l.t.e. results for data that are outside the Sobolev functional setting. More
precisely, the methods that we employ here will allow us to construct
solutions of some $abcd$ systems which live on time intervals of order
$O\left(  \varepsilon^{-1}\right)  $ where, for the $1$-dimensional case, we
might consider general disturbances modeled by continuous
functions\footnote{with some extra regularity on its derivative} $\eta
_{0}^{1D}$ such that:%
\[
\lim_{x\rightarrow\pm\infty}\eta_{0}^{1D}\left(  x\right)  =\eta_{\pm}.
\]
In the two dimensional setting, the situation in view is that of a
$2$-dimensional perturbation of the essentially $1$-dimensional situation,
more precisely we consider:%
\[
\left\{
\begin{array}
[c]{c}%
\eta_{0}^{2D}\left(  x,y\right)  =\eta_{0}^{1D}\left(  x\right)  +\phi\left(
x,y\right)  ,\\
V_{0}^{2D}\left(  x,y\right)  =\left(  u_{0}^{1D}\left(  x\right)  +\psi
_{1}\left(  x,y\right)  ,\psi_{2}\left(  x,y\right)  \right)
\end{array}
\right.
\]
where we ask%
\[
\lim_{\left\vert \left(  x,y\right)  \right\vert \rightarrow\infty}\left\vert
\phi\left(  x,y\right)  \right\vert =0\text{ and }\lim_{\left\vert \left(
x,y\right)  \right\vert \rightarrow\infty}\left\vert \psi_{i}\left(
x,y\right)  \right\vert =0\text{ for }i=1,2.
\]
\ $\ $Let us point out that in some sense, the change of the functional
setting in order to study bore propagation corresponds in solving a Neumann
problem. Indeed, solving $\left(  \text{\ref{eq1}}\right)  $ with an initial
data in $\mathcal{C}_{0}$ amounts in solving a Dirichlet-type problem whereas
solving it with initial data which is bounded with its gradient belonging to
$\mathcal{C}_{0}$ amounts to solving a Neumann-type problem.

\subsection{The main results}

In the present paper we will focus our attention on the so called BBM-type
Boussinesq systems of equations:
\begin{equation}
\left\{
\begin{array}
[c]{l}%
\left(  I-\varepsilon b\Delta\right)  \partial_{t}\bar{\eta}%
+\operatorname{div}\bar{V}+\varepsilon\operatorname{div}\left(  \bar{\eta}%
\bar{V}\right)  =0,\\
\left(  I-\varepsilon d\Delta\right)  \partial_{t}\bar{V}+\nabla\bar{\eta
}+\varepsilon\bar{V}\cdot\nabla\bar{V}=0,
\end{array}
\right.  \label{weaklydispersive1}%
\end{equation}
where the parameters $b$, $d$ obey:%
\[
b,d\geq0.
\]
We address the long time existence problem for the general $\left(
\text{\ref{weaklydispersive1}}\right)  $ system with initial data that can be
used to model bore-type waves. Before stating our results, let us fix the
functional framework where we are going to construct our solutions.

We fix two functions $\chi$ and $\varphi$ satisfying:%
\[
\forall\xi\in\mathbb{R}^{n}\text{, \ }\chi(\xi)+\sum_{j\geq0}\varphi(2^{-j}%
\xi)=1\text{,}%
\]
(see Proposition \ref{diadic} from the Appendix) and let us denote by $h$
respectively $\tilde{h}$ their Fourier inverses. For all $u\in\mathcal{S}%
^{\prime}$, the nonhomogeneous dyadic blocks are defined as follows:%
\begin{equation}
\left\{
\begin{array}
[c]{l}%
\Delta_{j}u=0\text{ \ if \ }j\leq-2,\\
\Delta_{-1}u=\chi\left(  D\right)  u=\tilde{h}\star u,\\
\Delta_{j}u=\varphi\left(  2^{-j}D\right)  u=2^{jd}\int_{\mathbb{R}^{n}%
}h\left(  2^{q}y\right)  u\left(  x-y\right)  dy\text{ \ if \ }j\geq0.
\end{array}
\right.  \label{diadicNeomogen}%
\end{equation}
Let us define now the nonhomogeneous Besov spaces.

\begin{definition}
Let $s\in\mathbb{R}$, $\left(  p,r\right)  \in\left[  1,\infty\right]  $. The
Besov space $B_{p,r}^{s}\left(  \mathbb{R}^{n}\right)  $ is the set of
tempered distributions $u\in\mathcal{S}^{\prime}$ such that:%
\[
\left\Vert u\right\Vert _{B_{p,r}^{s}}:=\left\Vert \left(  2^{js}\left\Vert
\Delta_{j}u\right\Vert _{L^{p}}\right)  _{j\in\mathbb{Z}}\right\Vert
_{\ell^{r}(\mathbb{Z)}}<\infty.
\]

\end{definition}

\bigskip In all that follows, unless otherwise mentioned, the $j$-subscript
for a tempered distribution is reserved for denoting the frequency localized
distribution i.e.:%
\begin{equation}
u_{j}\overset{not.}{=}\Delta_{j}u\text{.} \label{notinspired}%
\end{equation}

\begin{remark}
Taking advantage of the Fourier-Plancherel theorem and using $\left(
\text{\ref{211}}\right)  $ one sees that the classical Sobolev spaces $H^{s}$
coincide with $B_{2,2}^{s}$.
\end{remark}

\begin{remark}
\label{observatieH}Let us suppose that $s\in\mathbb{R}^{+}\backslash
\mathbb{N}$. The space $B_{\infty,\infty}^{s}$ coincides with the H\"{o}lder
space $\mathcal{C}^{\left[  s\right]  ,s-\left[  s\right]  }$ of bounded
functions $u\in L^{\infty}$ whose derivatives of order $\left\vert
\alpha\right\vert \leq\left[  s\right]  $ are bounded and satisfy%
\[
\left\vert \partial^{\alpha}u\left(  x\right)  -\partial^{\alpha}u\left(
y\right)  \right\vert \leq C\left\vert x-y\right\vert ^{s-\left[  s\right]
}\text{ for }\left\vert x-y\right\vert \leq1.
\]

\end{remark}

For all $s\in\mathbb{R}$, we define\footnote{Here, we use the convention
$\operatorname*{sgn}\left(  x\right)  =\frac{x}{\left\vert x\right\vert }$ for
$x\not =0$ and $\operatorname*{sgn}\left(  0\right)  =0$.}:%
\begin{equation}
\left\{
\begin{array}
[c]{c}%
s_{b}=s+\operatorname*{sgn}\left(  b\right)  ,\\
s_{d}=s+\operatorname*{sgn}(c).
\end{array}
\right.  \label{relaties2}%
\end{equation}
Let us denote by
\[
X_{b,d,r}^{s}\left(  \mathbb{R}^{n}\right)  =B_{2,r}^{s_{b}}\left(
\mathbb{R}^{n}\right)  \times\left(  B_{2,r}^{s_{d}}\left(  \mathbb{R}%
^{n}\right)  \right)  ^{n}\text{.}%
\]
For any $\varepsilon>0$, we consider the norm:%
\[
\left\{
\begin{array}
[c]{c}%
\left\Vert \left(  \eta,V\right)  \right\Vert _{X_{b,d,r}^{s,\varepsilon}%
}=\left\Vert \left(  2^{js}U_{j}\left(  \eta,V\right)  \right)  _{j\in
\mathbb{Z}}\right\Vert _{\ell^{r}\left(  \mathbb{Z}\right)  }\text{ where }\\
U_{j}^{2}\left(  \eta,V\right)  =%
{\displaystyle\int\limits_{\mathbb{R}^{n}}}
\left(  \left\vert \eta_{j}\right\vert ^{2}+\varepsilon b\sum
\limits_{k=\overline{1,n}}\left\vert \partial_{k}\eta_{j}\right\vert ^{2}%
+\sum\limits_{k=\overline{1,n}}\left\vert V_{j}^{k}\right\vert ^{2}%
+\varepsilon d\sum\limits_{k,l=\overline{1,n}}\left\vert \partial_{l}V_{j}%
^{k}\right\vert ^{2}\right)  .
\end{array}
\right.
\]
The space%
\[
E_{b,d,r}^{s}\left(  \mathbb{R}^{n}\right)  =\left\{  \left(  \eta,V\right)
\in L^{\infty}\left(  \mathbb{R}^{n}\right)  \times\left(  L^{\infty}\left(
\mathbb{R}^{n}\right)  \right)  ^{n}:\left(  \partial_{k}\eta,\partial
_{k}V\right)  \in X_{b,d,r}^{s-1}\text{ }\forall k\in\overline{1,n}\right\}
\]
endowed with the norm%
\[
\left\Vert \left(  \eta,V\right)  \right\Vert _{E_{b,d,r}^{s,\varepsilon}%
}=\left\Vert \left(  \eta,V\right)  \right\Vert _{L^{\infty}}+\left(
\sum_{k\in\overline{1,n}}\left\Vert \left(  \partial_{k}\eta,\partial
_{k}V\right)  \right\Vert _{X_{b,d,r}^{s-1,\varepsilon}}^{2}\right)
^{\frac{1}{2}}%
\]
is a Banach space. An important aspect is that the space $E_{b,d,r}^{s}\left(
\mathbb{R}\right)  $ admits functions that manifest nontrivial behavior at
infinity, see Remark \ref{Obs}. Our first result, pertaining to the
$1$-dimensional case is formulated in the following theorem.

\begin{theorem}
\label{Teorema2}Let us consider $s\in\mathbb{R},$ $r\in\lbrack1,\infty)$ such
that $s>\frac{3}{2}$ or $s=\frac{3}{2}$ and $r=1$. Let us consider $\left(
\eta_{0},u_{0}\right)  \in E_{b,d,r}^{s}\left(  \mathbb{R}\right)  $. Then,
there exist two real numbers $\varepsilon_{0}$, $C$ both depending on $s,b,d,$
and on $\left\Vert \left(  \eta_{0},u_{0}\right)  \right\Vert _{E_{b,d,r}%
^{s,1}}$ and a numerical constant $\tilde{C}$ such that the following holds
true. For any $\varepsilon\leq\varepsilon_{0}$, System $\left(
\text{\ref{weaklydispersive1}}\right)  $ supplemented with the initial data
$\left(  \eta_{0},u_{0}\right)  $, admits an unique solution $\left(
\bar{\eta}^{\varepsilon},\bar{u}^{\varepsilon}\right)  \in\mathcal{C}\left(
\left[  0,\frac{C}{\varepsilon}\right]  ,E_{b,d,r}^{s}\right)  $. Moreover,
the following estimate holds true:%
\[
\sup_{t\in\left[  0,\frac{C}{\varepsilon}\right]  }\left\Vert \left(
\bar{\eta}^{\varepsilon}\left(  t\right)  ,\bar{u}^{\varepsilon}\left(
t\right)  \right)  \right\Vert _{E_{b,d,r}^{s,\varepsilon}}+\sup_{t\in\left[
0,\frac{C}{\varepsilon}\right]  }\left\Vert \partial_{t}\bar{\eta
}^{\varepsilon}\left(  t\right)  \right\Vert _{L^{\infty}}\leq\tilde
{C}\left\Vert \left(  \eta_{0},u_{0}\right)  \right\Vert _{E_{b,d,r}%
^{s,\varepsilon}}.
\]

\end{theorem}

\begin{remark}
\label{Obs}\bigskip Let us give an example of function that fits into our
framework but is not covered by previous works dedicated to the long time
existence problem. A reasonable initial data for modeling bores would be:%
\[
\eta_{0}\left(  x\right)  =\tanh\left(  x\right)  :=\frac{e^{x}-e^{-x}}%
{e^{x}+e^{-x}}.
\]
which is smooth and manifests nontrivial behavior at $\pm\infty$, i.e.%
\[
\lim_{x\rightarrow\infty}\eta_{0}\left(  x\right)  =1\text{ and }%
\lim_{x\rightarrow-\infty}\eta_{0}\left(  x\right)  =-1\text{.}%
\]
One can verify that $\tanh\in%
{\displaystyle\bigcap\limits_{s\geq0}}
E_{b,d,r}^{s,1}$ but does not belong to any Sobolev space $H^{s}$.
\end{remark}

Let us consider $\sigma\geq\frac{1}{2}$, $s\geq1$ with the convention that
whenever there is equality in one of the previous relations then $r=1$. We
denote by $M_{b,d,r}^{\sigma,s}$ the space of $\left(  \eta,V\right)
\in\mathcal{C}\left(  \mathbb{R}^{2}\right)  \times\left(  \mathcal{C}\left(
\mathbb{R}^{2}\right)  \right)  ^{2}$ such that there exists $\left(
\eta^{1D},V^{1D}\right)  \in E_{b,d,r}^{\sigma}\left(  \mathbb{R}\right)  $
and $\left(  \eta^{2D},V^{2D}\right)  \in X_{b,d,r}^{s}\left(  \mathbb{R}%
^{2}\right)  $ for which
\begin{equation}
\left\{
\begin{array}
[c]{c}%
\eta\left(  x,y\right)  =\eta^{1D}\left(  x\right)  +\eta^{2D}\left(
x,y\right)  ,\\
V\left(  x,y\right)  =\left(  V^{1D}\left(  x\right)  +V_{1}^{2D}\left(
x,y\right)  ,V_{2}^{2D}\left(  x,y\right)  \right)
\end{array}
\right.  \label{desc}%
\end{equation}
for all $\left(  x,y\right)  \in$ $\mathbb{R}^{2}$. Let us introduce
$i:E_{b,d,r}^{\sigma}\left(  \mathbb{R}\right)  \rightarrow\mathcal{C}\left(
\mathbb{R}^{2}\right)  \times\left(  \mathcal{C}\left(  \mathbb{R}^{2}\right)
\right)  ^{2}$ such that for all $\left(  x,y\right)  \in\mathbb{R}^{2}$ we
have:
\[
i\left(  \left(  \eta,u\right)  \right)  \left(  x,y\right)  =\left(
\eta\left(  x\right)  ,\left(  u\left(  x\right)  ,0\right)  \right)  .
\]
Of course, because of the fact $i\left(  E_{b,d,r}^{\sigma}\left(
\mathbb{R}\right)  \right)  \cap X_{b,d,r}^{s}\left(  \mathbb{R}^{2}\right)
=\left\{  0\right\}  $, the functions appearing in the decomposition $\left(
\text{\ref{desc}}\right)  $ are unique. For all $\varepsilon>0$, let us
consider on $M_{b,d,r}^{\sigma,s}$, the norm%
\[
\left\Vert \left(  \eta,V\right)  \right\Vert _{M_{b,d,r}^{\sigma
,s,\varepsilon}}=\left\Vert \left(  \eta^{1D},V^{1D}\right)  \right\Vert
_{E_{b,d,r}^{\sigma,\varepsilon}\left(  \mathbb{R}\right)  }+\left\Vert
\left(  \eta^{2D},V^{2D}\right)  \right\Vert _{X_{b,d,r}^{s,\varepsilon
}\left(  \mathbb{R}^{2}\right)  }.
\]
It is easy to see that $\left(  M_{b,d,r}^{\sigma,s},\left\Vert \cdot
\right\Vert _{M_{b,d,r}^{\sigma,s,\varepsilon}}\right)  $ is a Banach space.
We can now formulate the result pertaining to the $2$-dimensional case:

\begin{theorem}
\label{Teorema3}Let us consider $s,\sigma\in\mathbb{R},$ $r\in\lbrack
1,\infty)$ such that%
\begin{equation}
s>2\text{ or }s=2\text{ and }r=1. \label{s}%
\end{equation}
and%
\begin{equation}
\sigma>s+\frac{3}{2}\text{ and }\sigma-\frac{1}{2}\in\mathbb{R}\backslash
\mathbb{N}\text{.} \label{sigma}%
\end{equation}
Let us consider $\left(  \eta_{0},u_{0}\right)  \in E_{b,d,r}^{\sigma}\left(
\mathbb{R}\right)  $ and $\left(  \phi,\psi\right)  \in X_{b,d,r}^{s}\left(
\mathbb{R}^{2}\right)  $. Then, there exist two real numbers $\varepsilon_{0}%
$, $C$ both depending on $s,b,d,$ and on $\left\Vert \left(  \eta_{0}%
,u_{0}\right)  \right\Vert _{E_{b,d,r}^{s,1}}+\left\Vert \left(  \phi
,\psi\right)  \right\Vert _{X_{b,d,r}^{s,1}\left(  \mathbb{R}^{2}\right)  }$
and a numerical constant $\tilde{C}$ such that the following holds true. For
any $\varepsilon\leq\varepsilon_{0}$, system $\left(
\text{\ref{weaklydispersive1}}\right)  $ supplemented with the initial data%
\[
\left\{
\begin{array}
[c]{l}%
\bar{\eta}_{0}\left(  x,y\right)  =\eta_{0}\left(  x\right)  +\phi\left(
x,y\right)  ,\\
\bar{V}_{0}\left(  x,y\right)  =\left(  u_{0}\left(  x\right)  ,0\right)
+\psi\left(  x,y\right)  ,
\end{array}
\right.
\]
admits an unique solution $\left(  \bar{\eta}^{\varepsilon},\bar
{V}^{\varepsilon}\right)  \in\mathcal{C}\left(  \left[  0,\frac{C}%
{\varepsilon}\right]  ,M_{b,d,r}^{\sigma,s}\right)  $. Moreover, the following
estimate holds true:%
\[
\sup_{t\in\left[  0,\frac{C}{\varepsilon}\right]  }\left(  \left\Vert \left(
\bar{\eta}^{\varepsilon}\left(  t\right)  ,\bar{V}^{\varepsilon}\left(
t\right)  \right)  \right\Vert _{M_{b,d,r}^{\sigma,s,\varepsilon}}\right)
+\sup_{t\in\left[  0,\frac{C}{\varepsilon}\right]  }\left(  \left\Vert
\partial_{t}\bar{\eta}^{\varepsilon}\left(  t\right)  \right\Vert _{L^{\infty
}}\right)  \leq\tilde{C}\left(  \left\Vert \left(  \eta_{0},u_{0}\right)
\right\Vert _{E_{b,d,r}^{\sigma,\varepsilon}}+\left\Vert \left(  \phi
,\psi\right)  \right\Vert _{X_{b,d,r}^{s,\varepsilon}}\right)
\]

\end{theorem}

The results presented in Theorems \ref{Teorema2}, \ref{Teorema3} are a
by-product of a general result that we obtain later in the paper (see Theorem
\ref{Teorema1} ). The method of proof consists of conveniently spliting the
initial data into two parts and of performing energy estimates on a slightly
more general system than $\left(  \text{\ref{weaklydispersive1}}\right)  $.
The estimates are a refined version of those obtained in \cite{Burtea1} and
they allow us to handle the fact that the bore type functions do not belong to
$L^{2}$.

\subsection{Notations}

Let us introduce some notations. For any vector field $U:\mathbb{R}%
^{n}\rightarrow\mathbb{R}^{n}$ we denote by $\nabla U:\mathbb{R}%
^{n}\rightarrow\mathcal{M}_{n}(\mathbb{R})$ and by the $n\times n$ matrices
defined by:%
\[
\left(  \nabla U\right)  _{ij}=\partial_{i}U^{j},
\]
In the same manner we define $\nabla^{2}U:\mathbb{R}^{n}\rightarrow
\mathbb{R}^{n}\times\mathbb{R}^{n}\times\mathbb{R}^{n}$ as:%
\[
\left(  \nabla^{2}U\right)  _{ijk}=\partial_{ij}^{2}U^{k}.
\]
We will suppose that all vectors appearing are column vectors and thus the
(classical) product between a matrix field $A$ and a vector field $U$ will be
the vector\footnote{From now on we will use the Einstein summation convention
over repeted indices.}:%
\[
\left(  AU\right)  ^{i}=A_{ij}U^{j}.
\]
We will often write the contraction operation between $\nabla^{2}U$ and a
vector field $V$ by%
\[
\left(  \nabla^{2}U:V\right)  _{ij}=\partial_{ij}^{2}U^{k}V^{k}%
\]
If $U,V:\mathbb{R}^{n}\rightarrow\mathbb{R}^{n}$ are two vector fields and
$A,B:\mathbb{R}^{n}\rightarrow\mathcal{M}_{n}(\mathbb{R})$ two matrix fields
we denote:%
\begin{align*}
UV  &  =U^{i}V^{i},\text{ }A:B=A_{ij}B_{ij},\\
\left\langle U,V\right\rangle _{L^{2}}  &  =\int U^{i}V^{i},\text{
}\left\langle A,B\right\rangle _{L^{2}}=\int A_{ij}B_{ij}\\
\left\Vert U\right\Vert _{L^{2}}^{2}  &  =\left\langle U,U\right\rangle
_{L^{2}},\text{ }\left\Vert A\right\Vert _{L^{2}}^{2}=\left\langle
A,A\right\rangle _{L^{2}}\\
\left\Vert \nabla^{2}U\right\Vert _{L^{2}}^{2}  &  =\int\nabla U:\nabla
U=\int\left(  \partial_{ij}U^{k}\right)  ^{2}%
\end{align*}
Also, the tensorial product of two vector fields $U,V$ is defined as the
matrix field $U\otimes V:\mathbb{R}^{n}\rightarrow\mathcal{M}_{n}(\mathbb{R})$
given by:%
\[
\left(  U\otimes V\right)  _{ij}=U^{i}V^{j}.
\]

\section{Some intermediate results}

The method of proof of the main results naturally leads us to study the
following system:%
\begin{equation}
\left\{
\begin{array}
[c]{l}%
\left(  I-\varepsilon b\Delta\right)  \partial_{t}\eta+\operatorname{div}%
V+\varepsilon\operatorname{div}\left(  \eta W_{1}+hV+\beta\eta V\right)
=\varepsilon f,\\
\left(  I-\varepsilon d\Delta\right)  \partial_{t}V+\nabla\eta+\varepsilon
\left(  W_{2}+\beta V\right)  \cdot\nabla V+\varepsilon V\cdot\nabla
W_{3}=\varepsilon g\\
\eta_{|t=0}=\eta_{0},\text{ }V_{|t=0}=V_{0}%
\end{array}
\right.  \tag{$\mathcal{S}_\varepsilon\left(     \mathcal{D}\right)
$}\label{BBM}%
\end{equation}
where $\varepsilon,\beta\in\left[  0,1\right]  $ and $\mathcal{D=}\left(
\eta_{0},V_{0},f,g,h,W_{1},W_{2},W_{3}\right)  $. The above system captures a
very general form of weakly dispersive quasilinear systems. Let us consider
$s\in\mathbb{R}$ such that
\begin{equation}
s>\frac{n}{2}+1\text{ or }s=\frac{n}{2}+1\text{ and }r=1. \label{relaties}%
\end{equation}
Let us fix the following notations:%

\begin{equation}
\left\{
\begin{array}
[c]{l}%
U_{j}^{2}(t)=\left\Vert \left(  \eta_{j}(t),V_{j}\left(  t\right)  \right)
\right\Vert _{L^{2}}^{2}+\varepsilon\left\Vert \left(  \sqrt{b}\nabla\eta
_{j}\left(  t\right)  ,\sqrt{d}\nabla V_{j}\left(  t\right)  \right)
\right\Vert _{L^{2}}^{2},\\
U_{s}\left(  t\right)  =\left\Vert \left(  2^{js}U_{j}\left(  t\right)
\right)  _{j\in\mathbb{Z}}\right\Vert _{\ell^{r}(\mathbb{Z})},U\left(
t\right)  =\left\Vert \left(  \eta,\nabla\eta,V,\nabla V\right)  \right\Vert
_{L^{\infty}\cap L^{p_{2}}},\text{ }F_{s}\left(  t\right)  =\left\Vert
(f(t),g(t))\right\Vert _{B_{2,r}^{s}},\\
\mathcal{W}_{s}\left(  t\right)  =\left\Vert h\left(  t\right)  \right\Vert
_{L^{\infty}}+\left\Vert \nabla h\left(  t\right)  \right\Vert _{B_{p_{1}%
,r}^{s}}+\left\Vert \partial_{t}h(t)\right\Vert _{L^{\infty}}+\sum_{i=1}%
^{3}\left(  \left\Vert \left(  W_{i}\left(  t\right)  \right)  \right\Vert
_{L^{\infty}}+\left\Vert \nabla W_{i}\left(  t\right)  \right\Vert
_{B_{p_{1},r}^{s}}\right)  .
\end{array}
\right.  \label{notatie}%
\end{equation}
In order to ease the reading we will rather skip denoting the time dependency
in the computations that follow. We are now in the position of stating the following:

\begin{theorem}
\label{Teorema1}Let us consider $b,d\geq0$, two real numbers with $b+d>0$,
$\left(  p_{1},r\right)  \in\left[  2,\infty\right]  $ $\times\lbrack
1,\infty)$, $p_{2}$ such that%
\[
\frac{1}{p_{1}}+\frac{1}{p_{2}}=\frac{1}{2}%
\]
and $s,s_{b},s_{d}>0$ as defined in $\left(  \text{\ref{relaties}}\right)  $
and $\left(  \text{\ref{relaties2}}\right)  $. Let us also consider $\left(
f,g\right)  \in\mathcal{C}\left(  [0,T],B_{2,r}^{s}\times\left(  B_{2,r}%
^{s}\right)  ^{n}\right)  $, $h\in\mathcal{C}\left(  [0,T],L^{\infty}\right)
$ with $\nabla h\in\mathcal{C}\left(  [0,T],B_{p_{1},r}^{s}\right)  $,
$\partial_{t}h\in L^{\infty}\left(  [0,T],L^{\infty}\right)  $ and for
$i=\overline{1,3}$ consider the vector fields $W_{i}\in\mathcal{C}\left(
[0,T],\left(  L^{\infty}\right)  ^{n}\right)  $ with $\nabla W_{i}%
\in\mathcal{C}\left(  [0,T],\left(  B_{p_{1},r}^{s}\right)  ^{n}\right)  $.
Then, for all $\left(  \eta_{0},V_{0}\right)  \in B_{2,r}^{s_{b}}\times\left(
B_{2,r}^{s_{d}}\right)  ^{n}$, writing \newline $\mathcal{D=}\left(  \eta_{0}%
,V_{0},f,g,h,W_{1},W_{2},W_{3}\right)  $, there exists a $\bar{T}\in(0,T]$
such that $\mathcal{S}_{\varepsilon}\left(  \mathcal{D}\right)  $ admits an
unique solution $\left(  \eta,V\right)  \in\mathcal{C}\left(  [0,\bar
{T}],B_{2,r}^{s_{b}}\times\left(  B_{2,r}^{s_{d}}\right)  ^{n}\right)  $ with
$\left(  \partial_{t}\eta,\partial_{t}V\right)  \in\mathcal{C}\left(
[0,\bar{T}],B_{2,r}^{s-1+2\operatorname*{sgn}b}\times\left(  B_{2,r}%
^{s-1+2\operatorname*{sgn}d}\right)  ^{n}\right)  $. Moreover if $T^{\star}%
\in\mathbb{R}^{+}$ is such that%
\begin{equation}
\int_{0}^{T^{\star}}U\left(  \tau\right)  d\tau<\infty, \label{Thm1.3}%
\end{equation}
then the solution $\left(  \eta,V\right)  $ can be continued after $T^{\star}$.
\end{theorem}

Of course, in what the existence and uniqueness of solutions of $\mathcal{S}%
_{\varepsilon}\left(  \mathcal{D}\right)  $ is concerned, the results
presented in Theorem \ref{Teorema1} are not optimal. However, the aspect that
we wish to emphasize in this paper is the possibility of solving the above
system on what we named long time scales and with regards to this matter the
extra regularity is necessary in order to develop the forthcoming theory. The
next result is the main ingredient in proving the l.t.e. results announced in
Theorem \ref{Teorema2} and Theorem \ref{Teorema3}.

\begin{theorem}
\label{Teorema1.1}Let us fix $\left(  \eta_{0},V_{0}\right)  \in X_{b,d,r}%
^{s}$. For $\varepsilon\in(0,1]$ we consider $\left(  f^{\varepsilon
},g^{\varepsilon}\right)  \in\mathcal{C}\left(  [0,T^{\varepsilon}%
],B_{2,r}^{s}\times\left(  B_{2,r}^{s}\right)  ^{n}\right)  $, $h^{\varepsilon
}\in\mathcal{C}\left(  [0,T^{\varepsilon}],L^{\infty}\right)  $ with $\nabla
h^{\varepsilon}\in\mathcal{C}\left(  [0,T^{\varepsilon}],B_{p_{1},r}%
^{s}\right)  $, $\partial_{t}h^{\varepsilon}\in L^{\infty}\left(
[0,T^{\varepsilon}],L^{\infty}\right)  $ and for $i=\overline{1,3}$ we
consider the vector fields $W_{i}^{\varepsilon}\in\mathcal{C}\left(
[0,T^{\varepsilon}],\left(  L^{\infty}\right)  ^{n}\right)  $ with $\nabla
W_{i}^{\varepsilon}\in\mathcal{C}\left(  [0,T^{\varepsilon}],\left(
B_{p_{1},r}^{s}\right)  ^{n}\right)  $. Assume that%
\[
\sup\limits_{\varepsilon\in\left[  0,1\right]  }\sup\limits_{\tau\in
\lbrack0,T^{\varepsilon}]}\mathcal{W}_{s}^{\varepsilon}\left(  \tau\right)
<\infty,\text{ }\sup\limits_{\varepsilon\in\left[  0,1\right]  }%
\sup\limits_{\tau\in\lbrack0,T^{\varepsilon}]}F_{s}^{\varepsilon}\left(
\tau\right)  <\infty,
\]
where%
\[
\mathcal{W}_{s}^{\varepsilon}\left(  t\right)  =\left\Vert h^{\varepsilon
}\left(  t\right)  \right\Vert _{L^{\infty}}+\left\Vert \nabla h^{\varepsilon
}\left(  t\right)  \right\Vert _{B_{p_{1},r}^{s}}+\left\Vert \partial
_{t}h^{\varepsilon}(t)\right\Vert _{L^{\infty}}+\sum_{i=1}^{3}\left(
\left\Vert \left(  W_{i}^{\varepsilon}\left(  t\right)  \right)  \right\Vert
_{L^{\infty}}+\left\Vert \nabla W_{i}^{\varepsilon}\left(  t\right)
\right\Vert _{B_{p_{1},r}^{s}}\right)  ,
\]
and%
\[
F_{s}^{\varepsilon}\left(  t\right)  =\left\Vert (f^{\varepsilon
}(t),g^{\varepsilon}(t))\right\Vert _{B_{2,r}^{s}}.
\]
We denote by $\mathcal{D}^{\varepsilon}\mathcal{=}\left(  \eta_{0}%
,V_{0},f^{\varepsilon},g^{\varepsilon},h^{\varepsilon},W_{1}^{\varepsilon
},W_{2}^{\varepsilon},W_{3}^{\varepsilon}\right)  $.Then, there exist two real
numbers $\varepsilon_{0}$, $C$ both depending on $s,n,b,d,\left\Vert \left(
\eta_{0},V_{0}\right)  \right\Vert _{X_{b,d,r}^{s,1}}$, $\sup
\limits_{\varepsilon\in\left[  0,1\right]  }\sup\limits_{\tau\in
\lbrack0,T^{\varepsilon}]}\mathcal{W}_{s}^{\varepsilon}\left(  \tau\right)  $,
$\sup\limits_{\varepsilon\in\left[  0,1\right]  }\sup\limits_{\tau\in
\lbrack0,T^{\varepsilon}]}F_{s}^{\varepsilon}\left(  \tau\right)  $ and a
numerical constant $\tilde{C}=\tilde{C}\left(  n\right)  $ such that the
following holds true. For any $\varepsilon\leq\varepsilon_{0}$, the maximal
time of existence $T_{\max}^{\varepsilon}$ of the unique solution $\left(
\eta^{\varepsilon},V^{\varepsilon}\right)  $ of system $\mathcal{S}%
_{\varepsilon}\left(  \mathcal{D}^{\varepsilon}\right)  $ satisfies the
following lower bound:%
\begin{equation}
T_{\max}^{\varepsilon}\geq T_{\star}^{\varepsilon}:\overset{def.}{=}%
\min\left\{  T^{\varepsilon},\frac{C}{\varepsilon}\right\}  . \label{Timp}%
\end{equation}
Moreover, we have that:%
\begin{equation}
\sup_{t\in\left[  0,T_{\star}^{\varepsilon}\right]  }\left\Vert \left(
\eta^{\varepsilon}\left(  t\right)  ,V^{\varepsilon}\left(  t\right)  \right)
\right\Vert _{X_{b,d,r}^{s,\varepsilon}}+\sup_{t\in\left[  0,T_{\star
}^{\varepsilon}\right]  }\left\Vert \partial_{t}\eta^{\varepsilon}\left(
t\right)  \right\Vert _{L^{\infty}}\leq\tilde{C}\left(  n\right)  \left\Vert
\left(  \eta_{0},V_{0}\right)  \right\Vert _{X_{b,d,r}^{s,\varepsilon}}.
\label{uniform}%
\end{equation}

\end{theorem}

Of course, when the data $\mathcal{D}^{\varepsilon}$ does not depend on
$\varepsilon$ we obtain the following:

\begin{corollary}
\label{Corolar}Let us fix $\left(  \eta_{0},V_{0}\right)  \in X_{b,d,r}^{s}$
and $s$ and $p_{1}$ as above. Also, consider $\left(  f,g\right)
\in\mathcal{C}\left(  [0,T],B_{2,r}^{s}\times\left(  B_{2,r}^{s}\right)
^{n}\right)  $, $h\in\mathcal{C}\left(  [0,T],L^{\infty}\right)  $ with
$\nabla h\in\mathcal{C}\left(  [0,T],B_{p_{1},r}^{s}\right)  $, $\partial
_{t}h\in L^{\infty}\left(  [0,T],L^{\infty}\right)  $ and for $i=\overline
{1,3}$ consider the vector fields $W_{i}\in\mathcal{C}\left(  [0,T],\left(
L^{\infty}\right)  ^{n}\right)  $ with $\nabla W_{i}\in\mathcal{C}\left(
[0,T],\left(  B_{p_{1},r}^{s}\right)  ^{n}\right)  $. We denote by
$\mathcal{D=}\left(  \eta_{0},V_{0},f,g,h,W_{1},W_{2},W_{3}\right)  $.Then,
there exist two real numbers $\varepsilon_{0}$, $C$ both depending on $s$%
,$n$,$b$,$d$,$\left\Vert \left(  \eta_{0},V_{0}\right)  \right\Vert
_{X_{b,d,r}^{s,1}}$, $\sup\limits_{\tau\in\lbrack0,T]}\mathcal{W}_{s}\left(
\tau\right)  $, $\sup\limits_{\tau\in\lbrack0,T]}F_{s}\left(  \tau\right)  $
and a numerical constant $\tilde{C}=\tilde{C}\left(  n\right)  $ such that the
following holds true. For any $\varepsilon\leq\varepsilon_{0}$, the maximal
time of existence $T_{\max}^{\varepsilon}$ of the unique solution $\left(
\eta^{\varepsilon},V^{\varepsilon}\right)  $ of the system $\mathcal{S}%
_{\varepsilon}\left(  \mathcal{D}\right)  $ satisfies the following lower
bound:%
\begin{equation}
T_{\max}^{\varepsilon}\geq T_{\star}^{\varepsilon}:\overset{def.}{=}%
\min\left\{  T^{\varepsilon},\frac{C}{\varepsilon}\right\}  . \label{Corolar1}%
\end{equation}
Moreover, we have that:%
\begin{equation}
\sup_{t\in\left[  0,T_{\star}^{\varepsilon}\right]  }\left\Vert \left(
\eta^{\varepsilon}\left(  t\right)  ,V^{\varepsilon}\left(  t\right)  \right)
\right\Vert _{X_{b,d,r}^{s,\varepsilon}}+\sup_{t\in\left[  0,T_{\star
}^{\varepsilon}\right]  }\left\Vert \partial_{t}\eta^{\varepsilon}\left(
t\right)  \right\Vert _{L^{\infty}}\leq\tilde{C}\left(  n\right)  \left\Vert
\left(  \eta_{0},V_{0}\right)  \right\Vert _{X_{b,d,r}^{s,\varepsilon}}.
\label{Corolar2}%
\end{equation}

\end{corollary}

\begin{remark}
The choice of $s$ according to relation $\left(  \text{\ref{relaties}}\right)
$ ensures that we have the following embedding: $B_{2,r}^{s}\hookrightarrow
L^{p_{2}}$ and $B_{2,r}^{s}\hookrightarrow L^{\infty}$. In particular, we also
have%
\[
U\left(  t\right)  \leq_{n}U_{s}\left(  t\right)  ,
\]
a fact that will be systematically used in all that follows.
\end{remark}

\begin{remark}
The explosion criterion $\left(  \text{\ref{Thm1.3}}\right)  $ implies that
the life span of the solution does not depend on its possible extra regularity
above the critical level $B_{2,1}^{\frac{n}{2}+1}$.
\end{remark}

The plan of the proof of Theorem \ref{Teorema1} is the following. First, we
derive a priori estimates using a spectral localization of the system $\left(
\text{\ref{BBM}}\right)  $. Then, we use the so called Friedrichs method in
order to construct a sequence of functions that solves a family of ODE's which
approximate system $\left(  \text{\ref{BBM}}\right)  $. Finally, using a
compactness method we show that we can construct a solution of the system
$\left(  \text{\ref{BBM}}\right)  $. Theorem \ref{Teorema1.1} is obtained
using a bootstrap argument.

\subsection{Proof of Theorem \ref{Teorema1}}

\subsubsection{A priori estimates}

Firs of all we will derive a priori estimates. Thus, let us consider $\left(
\eta,V\right)  \in\mathcal{C}\left(  [0,\bar{T}],B_{2,r}^{s_{b}}\times\left(
B_{2,r}^{s_{d}}\right)  ^{n}\right)  $ a solution of $\left(  \text{\ref{BBM}%
}\right)  $. As announced, we proceed by localizing the system $\left(
\text{\ref{BBM}}\right)  $ in the frequency space such that we obtain:%
\begin{equation}
\left\{
\begin{array}
[c]{l}%
\left(  I-\varepsilon b\Delta\right)  \partial_{t}\eta_{j}+\operatorname{div}%
V_{j}+\varepsilon\nabla\eta_{j}\left(  W_{1}+\beta V\right)  +\varepsilon
\left(  h+\beta\eta\right)  \operatorname{div}V_{j}=\varepsilon f_{j}%
+\varepsilon R_{1j},\\
\left(  I-\varepsilon d\Delta\right)  \partial_{t}V_{j}+\nabla\eta
_{j}+\varepsilon\left(  W_{2}+\beta V\right)  \cdot\nabla V_{j}=\varepsilon
g_{j}+\varepsilon R_{2j},
\end{array}
\right.  \label{ecj}%
\end{equation}
where%
\begin{equation}
\left\{
\begin{array}
[c]{l}%
R_{1j}=\beta\left[  V,\Delta_{j}\right]  \nabla\eta+\left[  W_{1},\Delta
_{j}\right]  \nabla\eta+\beta\left[  \eta,\Delta_{j}\right]
\operatorname{div}V+\left[  h,\Delta_{j}\right]  \operatorname{div}%
V-\Delta_{j}\left(  \nabla hV\right)  -\Delta_{j}\left(  \eta
\operatorname{div}W_{1}\right)  ,\\
R_{2j}=\left[  \left(  W_{2}+\beta V\right)  \cdot\nabla,\Delta_{j}\right]
V-\Delta_{j}\left(  V\cdot\nabla W_{3}\right)  ,
\end{array}
\right.  \label{Resturi}%
\end{equation}

We multiply the first equation of $\left(  \text{\ref{ecj}}\right)  $ with
$\eta_{j}$ and the second one with $1+\alpha\varepsilon\left(  h+\beta
\eta\right)  V_{j}$, we add the results and by integrating over $\mathbb{R}%
^{n}$ we obtain that:%
\[
\frac{1}{2}\frac{d}{dt}\left[  \int\eta_{j}^{2}+\varepsilon b\left\vert
\nabla\eta_{j}\right\vert ^{2}+\left(  1+\alpha\varepsilon\left(  h+\beta
\eta\right)  \right)  \left(  V_{j}^{2}+\varepsilon d\nabla V_{j}:\nabla
V_{j}\right)  \right]  =\sum_{i=1}^{7}T_{i}%
\]
where:%
\begin{align*}
T_{1}  &  =-\int\left(  1+\varepsilon\left(  h+\beta\eta\right)  \right)
\eta_{j}\operatorname{div}V_{j}-\int\left(  1+\alpha\varepsilon\left(
h+\beta\eta\right)  \right)  \nabla\eta_{j}V_{j},\\
T_{2}  &  =-\varepsilon\int\nabla\eta_{j}\left(  W_{1}+\beta V\right)
\eta_{j},\\
\text{ \ }T_{3}  &  =-\varepsilon\left\langle \left(  W_{2}+\beta V\right)
\cdot\nabla V_{j},\left(  1+\alpha\varepsilon\left(  h+\beta\eta\right)
\right)  V_{j}\right\rangle _{L^{2}},\\
T_{4}  &  =\varepsilon\int f_{j}\eta_{j}+\varepsilon\int\left(  1+\alpha
\varepsilon\left(  h+\beta\eta\right)  \right)  g_{j}V_{j},\\
T_{5}  &  =\varepsilon\int R_{1j}\eta_{j}+\varepsilon\int\left(
1+\alpha\varepsilon\left(  h+\beta\eta\right)  \right)  R_{2j}V_{j},\\
T_{6}  &  =-\frac{1}{2}\alpha\varepsilon\left\langle V_{j},\left(
\partial_{t}h+\beta\partial_{t}\eta\right)  V_{j}\right\rangle _{L^{2}}%
-\frac{1}{2}\alpha d\varepsilon^{2}\left\langle \nabla V_{j},\left(
\partial_{t}h+\beta\partial_{t}\eta\right)  \nabla V_{j}\right\rangle _{L^{2}%
},\\
T_{7}  &  =-\alpha d\varepsilon^{2}\left\langle \partial_{t}\nabla
V_{j},\nabla\left(  h+\beta\eta\right)  \otimes V_{j}\right\rangle _{L^{2}}.
\end{align*}
\ Here, we use $\alpha$ as a parameter in order to obtain the desired
estimates from a single "strike". Indeed only the values $\alpha\in\left\{
0,1\right\}  $ will be of interest to us and, as we will see, $\alpha=0$ will
give us the estimate necessary to develop the existence theory while
$\alpha=1$ will lead to an estimate that is the key point in finding the lower
bound on the time of existence.

Let us estimate the $T_{i}$'s. Regarding the first term, we write that:%
\begin{align}
T_{1}  &  =-\int\left(  1+\varepsilon\left(  h+\beta\eta\right)  \right)
\eta_{j}\operatorname{div}V_{j}-\int\left(  1+\alpha\varepsilon\left(
h+\beta\eta\right)  \right)  \nabla\eta_{j}V_{j}\nonumber\\
&  =\alpha\varepsilon\int\left(  \nabla h+\beta\nabla\eta\right)  \eta
_{j}V_{j}-\varepsilon\left(  1-\alpha\right)  \int\left(  \beta\eta+h\right)
\eta_{j}\operatorname{div}V_{j}\nonumber\\
&  \leq\alpha\varepsilon C\left(  \left\Vert \nabla h\right\Vert _{L^{\infty}%
}+\beta\left\Vert \nabla\eta\right\Vert _{L^{\infty}}\right)  \left\Vert
\eta_{j}\right\Vert _{L^{2}}\left\Vert V_{j}\right\Vert _{L^{2}}\nonumber\\
&  +C\frac{\left(  1-\alpha\right)  \sqrt{\varepsilon}}{\max(\sqrt{b},\sqrt
{d})}U_{j}^{2}\left(  \mathcal{W}_{s}\mathcal{+}\beta U\right) \nonumber\\
&  \leq\alpha\varepsilon CU_{j}^{2}\left(  \mathcal{W}_{s}\mathcal{+}\beta
U\right)  +C\frac{\left(  1-\alpha\right)  \sqrt{\varepsilon}}{\max(\sqrt
{b},\sqrt{d})}U_{j}^{2}\left(  \mathcal{W}_{s}\mathcal{+}\beta U\right)  .
\label{T1.1}%
\end{align}

Let us bound the second term:%
\begin{align}
T_{2}  &  =-\varepsilon\int\left(  W_{1}+\beta V\right)  \nabla\eta_{j}%
\eta_{j}\leq\frac{\varepsilon}{2}\int\left(  \left\Vert \operatorname{div}%
W_{1}\right\Vert _{L^{\infty}}+\beta\left\Vert \operatorname{div}V\right\Vert
_{L^{\infty}}\right)  \eta_{j}^{2}\nonumber\\
&  \leq\frac{\varepsilon}{2}U_{j}^{2}\left(  \mathcal{W}_{s}\mathcal{+}\beta
U\right)  . \label{T2}%
\end{align}

Using the Einstein summation convention over repeated indices, the term
$T_{3}$ is treated as follows:%
\begin{align}
T_{3}  &  =-\varepsilon\left\langle \left(  W_{2}+\beta V\right)  \cdot\nabla
V_{j},\left(  1+\alpha\varepsilon\left(  \beta\eta+h\right)  \right)
V_{j}\right\rangle _{L^{2}}\nonumber\\
&  =-\varepsilon\int\left(  1+\alpha\varepsilon(\beta\eta+h)\right)  \left(
\beta V^{m}+W_{2}^{m}\right)  \partial_{m}V_{j}^{k}V_{j}^{k}\nonumber\\
&  =\frac{\varepsilon}{2}\int\partial_{m}\left(  \left(  1+\alpha
\varepsilon(\beta\eta+h)\right)  \left(  \beta V^{m}+W_{2}^{m}\right)
\right)  V_{j}^{k}V_{j}^{k}\nonumber\\
&  =\frac{\varepsilon}{2}\int\operatorname{div}\left(  \left(  1+\alpha
\varepsilon(\beta\eta+h)\right)  \left(  \beta V+W_{2}\right)  \right)
\left\vert V_{j}\right\vert ^{2}\nonumber\\
&  \leq\frac{\varepsilon}{2}U_{j}^{2}\left(  \left\Vert W_{2}\right\Vert
_{L^{\infty}}+\beta\left\Vert V\right\Vert _{L^{\infty}}+\alpha\varepsilon
\left\Vert \left(  \beta\eta,\beta\nabla\eta,h,\nabla h\right)  \right\Vert
_{L^{\infty}}\left\Vert \left(  \beta V,\beta\nabla V,W_{2},\nabla
W_{2}\right)  \right\Vert _{L^{\infty}}\right) \nonumber\\
&  \leq\frac{\varepsilon}{2}U_{j}^{2}\left(  \mathcal{W}_{s}+\beta
U+\alpha\varepsilon\left(  \mathcal{W}_{s}\mathcal{+}\beta U\right)
^{2}\right)  . \label{T3}%
\end{align}

Next, we bound the term $T_{4}:$%
\begin{align}
T_{4}  &  =\varepsilon\int f_{j}\eta_{j}+\varepsilon\int\left(  1+\alpha
\varepsilon\left(  h+\beta\eta\right)  \right)  g_{j}V_{j}\nonumber\\
&  \leq\varepsilon\left\Vert f_{j}\right\Vert _{L^{2}}\left\Vert \eta
_{j}\right\Vert _{L^{2}}+\varepsilon\left(  1+\alpha\varepsilon\left\Vert
h\right\Vert _{L^{\infty}}+\alpha\beta\varepsilon\left\Vert \eta\right\Vert
_{L^{\infty}}\right)  \left\Vert g_{j}\right\Vert _{L^{2}}\left\Vert
V_{j}\right\Vert _{L^{2}}\nonumber\\
&  \leq\varepsilon C2^{-js}c_{j}^{1}\left(  t\right)  U_{j}F_{s}\left(
1+\alpha\varepsilon\left(  \mathcal{W}_{s}+\beta U\right)  \right)  .
\label{T4}%
\end{align}

Treating the fifth term is done in the following lines. First we write that:%
\begin{align}
T_{5}  &  =\varepsilon\int R_{1j}\eta_{j}+\varepsilon\int\left(
1+\alpha\varepsilon\left(  h+\beta\eta\right)  \right)  R_{2j}V_{j}\nonumber\\
&  \leq\varepsilon CU_{j}\left(  1+\alpha\varepsilon\left(  \mathcal{W}%
_{s}+\beta U\right)  \right)  \left(  \left\Vert R_{1j}\right\Vert _{L^{2}%
}+\left\Vert R_{2j}\right\Vert _{L^{2}}\right)  . \label{Rest1}%
\end{align}
Next, let us estimate the remainder terms $R_{1j}$ and $R_{2j}$. In order to
do so, we apply Proposition \ref{comutneomogen} and Proposition \ref{produs}%
\ from the Appendix in order to get:%
\begin{align}
\left\Vert R_{1j}\right\Vert _{L^{2}}  &  \leq C2^{-js}c_{j}^{2}\left(
t\right)  \left\{  \beta\left\Vert \nabla V\right\Vert _{L^{\infty}}\left\Vert
\nabla\eta\right\Vert _{B_{2,r}^{s-1}}+\beta\left\Vert \nabla\eta\right\Vert
_{L^{\infty}}\left\Vert \nabla V\right\Vert _{B_{2,r}^{s-1}}+\right.
\nonumber\\
&  \text{ \ \ \ \ \ \ \ \ \ \ \ \ \ \ \ \ \ \ \ }\left\Vert \nabla
W_{1}\right\Vert _{L^{\infty}}\left\Vert \nabla\eta\right\Vert _{B_{2,r}%
^{s-1}}+\left\Vert \nabla\eta\right\Vert _{L^{p_{2}}}\left\Vert \nabla
W_{1}\right\Vert _{B_{p_{1},r}^{s-1}}+\nonumber\\
&  \text{ \ \ \ \ \ \ \ \ \ \ \ \ \ \ \ \ \ \ \ }\beta\left\Vert \nabla
\eta\right\Vert _{L^{\infty}}\left\Vert \nabla V\right\Vert _{B_{2,r}^{s-1}%
}+\beta\left\Vert \nabla V\right\Vert _{L^{\infty}}\left\Vert \nabla
\eta\right\Vert _{B_{2,r}^{s-1}}+\nonumber\\
&  \text{\ \ \ \ \ \ \ \ \ \ \ \ \ \ \ \ \ \ \ \ }\left\Vert \nabla
h\right\Vert _{L^{\infty}}\left\Vert \nabla V\right\Vert _{B_{2,r}^{s-1}%
}+\left\Vert \nabla V\right\Vert _{L^{p_{2}}}\left\Vert \nabla h\right\Vert
_{B_{p_{1},r}^{s-1}}+\nonumber\\
&  \text{ \ \ \ \ \ \ \ \ \ \ \ \ \ \ \ \ \ \ \ }\left\Vert \nabla
h\right\Vert _{L^{\infty}}\left\Vert \nabla V\right\Vert _{B_{2,r}^{s-1}%
}+\left\Vert V\right\Vert _{L^{p_{2}}}\left\Vert \nabla h\right\Vert
_{B_{p_{1},r}^{s}}+\nonumber\\
&  \text{ \ \ \ \ \ \ \ \ \ \ \ \ \ \ \ \ \ \ }\left.  \left\Vert
\operatorname{div}W_{1}\right\Vert _{L^{\infty}}\left\Vert \nabla
\eta\right\Vert _{B_{2,r}^{s-1}}+\left\Vert \eta\right\Vert _{L^{p_{2}}%
}\left\Vert \operatorname{div}W_{1}\right\Vert _{B_{p_{1},r}^{s}}\right\}
\label{Rest2}%
\end{align}
and consequently:%
\[
\left\Vert R_{1j}\right\Vert _{L^{2}}\leq C2^{-js}c_{j}^{2}\left(  t\right)
\left(  \mathcal{W}_{s}U+U_{s}\left(  \mathcal{W}_{s}+\beta U\right)  \right)
.
\]
Proceeding as above, we get a similar bound for $R_{2j}$. Thus, we get that:%
\begin{equation}
T_{5}\leq\varepsilon C2^{-js}c_{j}^{2}\left(  t\right)  U_{j}\left(
1+\alpha\varepsilon\left(  \mathcal{W}_{s}+\beta U\right)  \right)  \left(
\mathcal{W}_{s}U+U_{s}\left(  \mathcal{W}_{s}+\beta U\right)  \right)
\label{T5.11}%
\end{equation}

\bigskip When $\alpha=0$, $T_{6}=T_{7}=0$ and we are in the position of
obtaining the first estimate. Combining $\left(  \text{\ref{T1.1}}\right)  $,
$\left(  \text{\ref{T2}}\right)  $, $\left(  \text{\ref{T3}}\right)  $,
$\left(  \text{\ref{T4}}\right)  $ and $\left(  \text{\ref{T5.11}}\right)  $
we get that:%

\begin{equation}
\frac{d}{dt}U_{j}^{2}\leq\sqrt{\varepsilon}CU_{j}^{2}\left(  \mathcal{W}%
_{s}+\beta U\right)  +\varepsilon C2^{-js}c_{j}^{5}U_{j}\left(  F_{s}%
+\mathcal{W}_{s}U+U_{s}\left(  \mathcal{W}_{s}+\beta U\right)  \right)  .
\label{Gronwal0}%
\end{equation}
Time integration of $\left(  \text{\ref{Gronwal0}}\right)  $ reveals the
following:%
\begin{align*}
U_{j}\left(  t\right)   &  \leq U_{j}\left(  0\right)  +\sqrt{\varepsilon
}C\int_{0}^{t}U_{j}\left(  \tau\right)  \left(  \mathcal{W}_{s}\left(
\tau\right)  +\beta U\left(  \tau\right)  \right)  d\tau+\\
&  \sqrt{\varepsilon}C\int_{0}^{t}2^{-j\sigma}c_{j}^{5}\left(  \tau\right)
\left(  F_{s}\left(  \tau\right)  +\mathcal{W}_{s}\left(  \tau\right)
U\left(  \tau\right)  +U_{s}\left(  \tau\right)  \left(  \mathcal{W}%
_{s}\left(  \tau\right)  +\beta U\left(  \tau\right)  \right)  \right)  d\tau.
\end{align*}
Multiplying the last inequality with $2^{js}$ and performing an $\ell
^{r}\left(  \mathbb{Z}\right)  $-summation we end up with:%
\begin{align}
U_{s}\left(  t\right)   &  \leq U_{s}\left(  0\right)  +\sqrt{\varepsilon
}C\int_{0}^{t}F_{s}\left(  \tau\right)  d\tau+\sqrt{\varepsilon}C\int_{0}%
^{t}\mathcal{W}_{s}\left(  \tau\right)  U\left(  \tau\right)  d\tau\nonumber\\
&  +\sqrt{\varepsilon}C\int_{0}^{t}U_{s}\left(  \tau\right)  \left(
\mathcal{W}_{s}\left(  \tau\right)  +\beta U\left(  \tau\right)  \right)
d\tau.\nonumber\\
&  \leq U_{s}\left(  0\right)  +\sqrt{\varepsilon}C\int_{0}^{t}F_{s}\left(
\tau\right)  d\tau+\sqrt{\varepsilon}C\int_{0}^{t}U_{s}\left(  \tau\right)
\left(  \mathcal{W}_{s}\left(  \tau\right)  +\beta U\left(  \tau\right)
\right)  d\tau. \label{Gron3}%
\end{align}
Obviously, relation $\left(  \text{\ref{Gron3}}\right)  $ and Gronwall's lemma
imply the explosion criteria.\bigskip

Let us now bound the term $T_{6}$:%
\begin{align*}
T_{6}  &  =-\frac{1}{2}\alpha\varepsilon\left\langle V_{j},\left(
\partial_{t}h+\beta\partial_{t}\eta\right)  V_{j}\right\rangle _{L^{2}}%
-\frac{1}{2}\alpha\varepsilon^{2}d\left\langle \nabla V_{j},\left(
\partial_{t}h+\beta\partial_{t}\eta\right)  \nabla V_{j}\right\rangle _{L^{2}%
}\\
&  \leq\alpha\varepsilon U_{j}^{2}\left(  \left\Vert \partial_{t}h\right\Vert
_{L^{\infty}}+\beta\left\Vert \partial_{t}\eta\right\Vert _{L^{\infty}%
}\right)  .
\end{align*}
Using the first equation of $\left(  \text{\ref{BBM}}\right)  $ we write:%
\[
\partial_{t}\eta=\varepsilon\left(  I-\varepsilon b\Delta\right)
^{-1}f-\left(  I-\varepsilon b\Delta\right)  ^{-1}\operatorname{div}\left(
V+\varepsilon\left(  \eta W_{1}+hV+\beta\eta V\right)  \right)
\]
and using again relation $\left(  \text{\ref{relaties}}\right)  $ combined
with the fact that $\left(  I-\varepsilon b\Delta\right)  ^{-1}$ has at most
norm $1$ when regarded as an $L^{2}$ to $L^{2}$ operator, we obtain that:%
\begin{align}
\left\Vert \partial_{t}\eta\right\Vert _{L^{\infty}}  &  \leq\left\Vert
\varepsilon\left(  I-\varepsilon b\Delta\right)  ^{-1}f-\left(  I-\varepsilon
b\Delta\right)  ^{-1}\operatorname{div}\left(  V+\varepsilon\left(  \eta
W_{1}+hV+\beta\eta V\right)  \right)  \right\Vert _{B_{2,1}^{\frac{n}{2}}%
}\nonumber\\
&  \leq\varepsilon\left\Vert f\right\Vert _{B_{2,r}^{s}}+\left\Vert
V\right\Vert _{B_{2,r}^{s}}+\varepsilon\left\Vert \eta W_{1}+hV+\beta\eta
V\right\Vert _{B_{2,r}^{s}}\nonumber\\
&  \leq\varepsilon\left\Vert f\right\Vert _{B_{2,r}^{s}}+\left\Vert
V\right\Vert _{B_{2,r}^{s}}+\varepsilon\beta\left\Vert \left(  \eta,V\right)
\right\Vert _{B_{2,r}^{s}}^{2}+\varepsilon\left\Vert \left(  \eta,V\right)
\right\Vert _{B_{2,r}^{s}}\left\Vert \left(  h,W_{1}\right)  \right\Vert
_{L^{\infty}}\nonumber\\
&  +\varepsilon\left\Vert \left(  \eta,V\right)  \right\Vert _{L^{p_{2}}%
}\left\Vert \left(  \nabla h,\nabla W_{1}\right)  \right\Vert _{B_{p_{1}%
,r}^{s-1}}\nonumber\\
&  \leq C(U_{s}+\varepsilon F_{s}+\varepsilon U_{s}\left(  \mathcal{W}%
_{s}+\beta U_{s}\right)  ). \label{derivatatemporala}%
\end{align}
Thus, putting togeter the last estimates, we find that:%
\begin{equation}
T_{6}\leq\alpha\varepsilon U_{j}^{2}\left(  \mathcal{W}_{s}+\beta
U_{s}+\varepsilon\beta F_{s}+\varepsilon\left(  \mathcal{W}_{s}+\beta
U_{s}\right)  ^{2}\right)  . \label{T6}%
\end{equation}

Finally, let us estimate the last term:%
\[
T_{7}=-\alpha d\varepsilon^{2}\left\langle \partial_{t}\nabla V_{j}%
,\nabla\left(  \beta\eta+h\right)  \otimes V_{j}\right\rangle _{L^{2}}%
\leq\alpha d\varepsilon^{2}\left\Vert \partial_{t}\nabla V_{j}\right\Vert
_{L^{2}}\left\Vert V_{j}\right\Vert _{L^{2}}\left(  \beta\left\Vert \nabla
\eta\right\Vert _{L^{\infty}}+\left\Vert \nabla h\right\Vert _{L^{\infty}%
}\right)  .
\]
Using the second equation of $\left(  \text{\ref{BBM}}\right)  $ we write
that:%
\begin{align*}
\partial_{t}\nabla V_{j}  &  =\varepsilon\left(  I-\varepsilon d\Delta\right)
^{-1}\nabla\left(  g_{j}\right)  -\left(  I-\varepsilon\Delta\right)
^{-1}\nabla^{2}\eta_{j}-\varepsilon\left(  I-\varepsilon\Delta\right)
^{-1}\nabla\Delta_{j}\left[  \left(  W_{2}+\beta V\right)  \cdot\nabla
V\right] \\
&  -\varepsilon\left(  I-\varepsilon\Delta\right)  ^{-1}\nabla\Delta
_{j}\left(  V\cdot\nabla W_{3}\right)
\end{align*}
and because $(\varepsilon d)^{\frac{1}{2}}\left(  I-\varepsilon d\Delta
\right)  ^{-1}\nabla$ and $\varepsilon d\left(  I-\varepsilon d\Delta\right)
^{-1}\nabla^{2}$ have $O\left(  1\right)  $-norms when regarded as operators
from $L^{2}$ to $L^{2}$, we can write that%
\begin{align*}
\varepsilon d\left\Vert \partial_{t}\nabla V_{j}\right\Vert _{L^{2}}  &  \leq
C\left(  \varepsilon^{\frac{3}{2}}\sqrt{d}\left\Vert g_{j}\right\Vert _{L^{2}%
}+\left\Vert \eta_{j}\right\Vert _{L^{2}}+\varepsilon^{\frac{3}{2}}\sqrt
{d}\left\Vert \Delta_{j}\left[  \left(  W_{2}+\beta V\right)  \cdot\nabla
V\right]  \right\Vert _{L^{2}}+\varepsilon^{\frac{3}{2}}\sqrt{d}\left\Vert
\Delta_{j}\left(  V\cdot\nabla W_{3}\right)  \right\Vert _{L^{2}}\right) \\
&  \leq C\max\left\{  1,\sqrt{\varepsilon d}\right\}  \left(  U_{j}%
+2^{-js}c_{j}^{4}\left(  \varepsilon F_{s}+\varepsilon\left(  \mathcal{W}%
_{s}+\beta U_{s}\right)  ^{2}\right)  \right)  .
\end{align*}
Thus, we get that:%
\begin{equation}
T_{7}\leq\alpha\varepsilon U_{j}^{2}\left(  \mathcal{W}_{s}+\beta U\right)
+\alpha\varepsilon C2^{-js}c_{j}^{4}\left(  t\right)  U_{j}\left(  \varepsilon
F_{s}\left(  \mathcal{W}_{s}+\beta U\right)  +\varepsilon U_{s}\left(
\mathcal{W}_{s}+\beta U_{s}\right)  ^{2}\right)  . \label{T7.1}%
\end{equation}

Let us consider $\alpha=1$. Putting together estimates $\left(
\text{\ref{T1.1}}\right)  $-$\left(  \text{\ref{T5.11}}\right)  $ we get that:%
\begin{gather}
\frac{d}{dt}\left[  \int\eta_{j}^{2}+\varepsilon b\left\vert \nabla\eta
_{j}\right\vert ^{2}+\left(  1+\varepsilon\left(  \eta+h\right)  \right)
\left(  V_{j}^{2}+\varepsilon d\nabla V_{j}:\nabla V_{j}\right)  \right]
\leq\varepsilon CU_{j}^{2}\left(  \varepsilon\beta F_{s}+\mathcal{W}_{s}+\beta
U_{s}+\varepsilon\left(  \mathcal{W}_{s}+\beta U_{s}\right)  ^{2}\right)
\nonumber\\
+\varepsilon C2^{-js}c_{j}\left(  t\right)  U_{j}\left(  F_{s}\left(
1+\varepsilon\left(  \mathcal{W}_{s}+\beta U_{s}\right)  \right)
+U_{s}\left(  \mathcal{W}_{s}+\beta U_{s}\right)  \left(  1+\varepsilon\left(
\mathcal{W}_{s}+\beta U_{s}\right)  \right)  \right)  . \label{final1}%
\end{gather}

\subsubsection{Existence and uniqueness of solutions}

We are now in the position to prove the existence and uniqueness of solutions
for system $\left(  \text{\ref{BBM}}\right)  $. We will use the so called
Friedrichs method. For all $m\in\mathbb{N}$, let us consider $\mathbb{E}_{m}$
the low frequency cut-off operator defined by:%
\[
\mathbb{E}_{m}f=\mathcal{F}^{-1}\left(  \chi_{B\left(  0,m\right)  }\hat
{f}\right)  .
\]
We define the space
\[
L_{m}^{2}=\left\{  f\in L^{2}:\text{\textrm{Supp}}\hat{f}\subset B\left(
0,m\right)  \right\}
\]
which, endowed with the $\left\Vert \cdot\right\Vert _{L^{2}}$-norm is a
Banach space. Let us observe that due to Bernstein's lemma, all Sobolev norms
are equivalent on $L_{m}^{2}$. For all $m\in\mathbb{N}$, we consider the
following differential equation on $L_{m}^{2}$:%
\begin{equation}
\left\{
\begin{array}
[c]{l}%
\partial_{t}\eta=F_{m}\left(  \eta,V\right)  ,\\
\partial V=G_{m}\left(  \eta,V\right)  ,\\
\eta_{|t=0}=\mathbb{E}_{m}\eta_{0},\text{ }V_{|t=0}=\mathbb{E}_{m}V_{0},
\end{array}
\right.  \label{En2}%
\end{equation}
where $\left(  F_{m},G_{m}\right)  :L_{m}^{2}\times\left(  L_{m}^{2}\right)
^{n}\rightarrow L_{m}^{2}\times\left(  L_{m}^{2}\right)  ^{n}$ are defined by:%
\begin{align}
F_{m}\left(  \eta,V\right)   &  =-\mathbb{E}_{m}\left(  \left(  I-\varepsilon
b\Delta\right)  ^{-1}\left[  \operatorname{div}V+\varepsilon\operatorname{div}%
\left(  \eta W_{1}+hV+\beta\eta V\right)  -\varepsilon f\right]  \right)
,\label{Fm}\\
G_{m}\left(  \eta,V\right)   &  =-\mathbb{E}_{m}\left(  \left(  I-\varepsilon
d\Delta\right)  ^{-1}\left[  \nabla\eta+\varepsilon\left(  W_{2}+\beta\right)
\cdot\nabla V+\varepsilon V\cdot\nabla W_{3}-\varepsilon g\right]  \right)  .
\label{Gm}%
\end{align}
It transpires that due to the equivalence of the Sobolev norm, $\left(
F_{m},G_{m}\right)  $ is continuous and locally Lipschitz on $L_{m}^{2}%
\times\left(  L_{m}^{2}\right)  ^{n}$. Thus, the classical Picard theorem
ensures that there exists a nonnegative time $T_{m}>0$ and a unique solution
$\left(  \eta^{m},V^{m}\right)  :\mathcal{C}^{1}\left(  [0,T_{m}],L_{m}%
^{2}\times\left(  L_{m}^{2}\right)  ^{n}\right)  $. Let us denote by
\[
\left\{
\begin{array}
[c]{l}%
U_{j}^{m}(t)=\left(  \left\Vert \left(  \Delta_{j}\eta^{m}(t),\Delta_{j}%
V^{m}\left(  t\right)  \right)  \right\Vert _{L^{2}}^{2}+\varepsilon\left\Vert
\left(  \sqrt{b}\nabla\Delta_{j}\eta^{m}\left(  t\right)  ,\sqrt{d}%
\nabla\Delta_{j}V^{m}\left(  t\right)  \right)  \right\Vert _{L^{2}}%
^{2}\right)  ^{\frac{1}{2}},\\
U_{s}^{m}\left(  t\right)  =\left\Vert \left(  2^{js}U_{j}^{m}\left(
t\right)  \right)  _{j\in\mathbb{Z}}\right\Vert _{\ell^{r}(\mathbb{Z})}%
,U^{m}\left(  t\right)  =\left\Vert \left(  \eta^{m},\nabla\eta^{m}%
,V^{m},\nabla V^{m}\right)  \right\Vert _{L^{\infty}\cap L^{p_{2}}}.
\end{array}
\right.
\]
Thanks to the property $\mathbb{E}_{m}^{2}=\mathbb{E}_{m}$, we get that the
estimate obtained in \eqref{Gron3} still holds true for $\left(  \eta
^{m},V^{m}\right)  $, namely:%
\begin{align*}
U_{s}^{m}\left(  t\right)   &  \leq U_{s}^{m}\left(  0\right)  +\sqrt
{\varepsilon}C\int_{0}^{t}F_{s}\left(  \tau\right)  d\tau+\sqrt{\varepsilon
}C\int_{0}^{t}U_{s}^{m}\left(  \tau\right)  \left(  U^{m}\left(  \tau\right)
+\mathcal{W}_{s}\left(  \tau\right)  \right)  d\tau\\
&  \leq U_{s}\left(  0\right)  +\sqrt{\varepsilon}C\int_{0}^{t}\left(
F_{s}\left(  \tau\right)  +\mathcal{W}_{s}^{2}\left(  \tau\right)  \right)
d\tau+\sqrt{\varepsilon}C\int_{0}^{t}\left(  U_{s}^{m}\left(  \tau\right)
\right)  ^{2}d\tau.
\end{align*}
We consider%
\[
T^{\prime}=\sup\left\{  T>0:\sqrt{\varepsilon}C\int_{0}^{t}\left(
F_{s}\left(  \tau\right)  +\mathcal{W}^{2}\left(  \tau\right)  \right)
d\tau\leq U_{s}\left(  0\right)  \right\}
\]
and%
\[
\bar{T}_{m}=\sup\left\{  T>0:U_{s}\left(  \tau\right)  \leq3U_{s}\left(
0\right)  \right\}  .
\]
Gronwall's lemma ensures the existence of a constant $\bar{C}=\bar{C}\left(
s,b,d\right)  $ such that:
\[
\bar{T}_{m}\geq\min\left\{  \frac{\bar{C}}{\sqrt{\varepsilon}U_{s}\left(
0\right)  },T^{\prime}\right\}  :=\bar{T}%
\]
and thus, the sequence $\left(  \eta^{m},V^{m}\right)  \in\mathcal{C}\left(
[0,T],B_{2,r}^{s_{1}}\times\left(  B_{2,r}^{s_{2}}\right)  ^{n}\right)  $ is
uniformly bounded i.e.%
\begin{equation}
\left\Vert \left(  \eta^{m},V^{m}\right)  \right\Vert _{E}\leq3U_{s}\left(
0\right)  . \label{Unif1}%
\end{equation}

From the relation%
\[
\partial_{t}\eta^{m}=-\mathbb{E}_{m}\left(  \left(  I-\varepsilon
b\Delta\right)  ^{-1}\left[  \operatorname{div}V^{m}+\varepsilon
\operatorname{div}\left(  \eta^{m}W_{1}+hV^{m}+\beta\eta^{m}V^{m}\right)
-\varepsilon f\right]  \right)
\]
and from $\left(  \text{\ref{Unif1}}\right)  $ we get that the sequence
$\left(  \partial_{t}\eta^{m}\right)  _{m\in\mathbb{N}}$ is uniformly bounded
in $B_{2,r}^{s-1}$ on $\left[  0,\bar{T}\right]  $. Next, considering for all
$p\in\mathbb{N}$ a smooth function $\phi_{p}$ such that:%
\[
\left\{
\begin{array}
[c]{c}%
\mathrm{Supp}\text{ }\phi_{p}\subset B\left(  0,p+1\right)  ,\\
\phi_{p}=1\text{ on }B\left(  0,p\right)
\end{array}
\right.
\]
it follows that for each $p\in\mathbb{N}$, the sequence $\left(  \phi_{p}%
\eta^{m}\right)  _{m\in\mathbb{N}}$ is uniformly equicontinuous on $\left[
0,\bar{T}\right]  $ and that for all $t\in\left[  0,\bar{T}\right]  $, the set
$\left\{  \phi_{p}\eta^{m}\left(  t\right)  :m\in\mathbb{N}\right\}  $ is
relatively compact in $B_{2,r}^{s-1}$. Thus, the Ascoli-Arzela Theorem
combined with Proposition \ref{compact} and with Cantor's diagonal process
provides us a subsequence of $\left(  \eta^{m}\right)  _{m\in\mathbb{N}}$ and
a tempered distribution $\eta\in\mathcal{C}\left(  [0,T],S^{\prime}\right)  $
such that for all $\phi\in\mathcal{D}\left(  \mathbb{R}^{n}\right)  $:%
\[
\phi\eta^{m}\rightarrow\phi\eta\text{ in }\mathcal{C}\left(  [0,T],B_{2,r}%
^{s-1}\right)  .
\]
Moreover, owing to the Fatou property for Besov spaces, see Proposition
\ref{PropBesov}, we get that $\eta\in L^{\infty}\left(  [0,T],B_{2,r}^{s_{1}%
}\right)  $ and thus, using interpolation we get that:%
\begin{equation}
\phi\eta^{m}\rightarrow\phi\eta\text{ in }\mathcal{C}\left(  [0,T],B_{2,r}%
^{s_{b}-\gamma}\right)  \label{conveta}%
\end{equation}
for any $\gamma>0$. Of course, using the same argument we can construct a
$V\in L^{\infty}\left(  \left[  0,\bar{T}\right]  ,\left(  B_{2,r}^{s_{2}%
}\right)  ^{n}\right)  $ such that for any $\psi\in\left(  \mathcal{D}\left(
\mathbb{R}^{n}\right)  \right)  ^{n}$:%
\begin{equation}
\psi V^{m}\rightarrow\psi V\text{ in }\mathcal{C}\left(  [0,T],\left(
B_{2,r}^{s_{d}-\gamma}\right)  ^{n}\right)  \label{convv}%
\end{equation}
for any $\gamma>0$. Also, by the Fatou property in Besov spaces we get that
$\left(  \eta,V\right)  \in L^{\infty}\left(  \left[  0,\bar{T}\right]
,B_{2,r}^{s_{b}}\times\left(  B_{2,r}^{s_{d}}\right)  ^{n}\right)  $. We claim
that the properties enlisted above allow us to pass to the limit when
$m\rightarrow\infty$ in the equation verified by $\eta^{m}$ and $V^{m}$. Let
us show on two examples, how this process is carried out in practice. Let
$\phi\in\mathcal{D}\left(  \mathbb{R}^{n}\right)  $ and let us write that:%
\begin{align}
\left\vert \int\phi\operatorname{div}\left[  \left(  \eta^{m}-\eta\right)
W_{1}\right]  \right\vert  &  =\left\vert \int\left(  \eta^{m}-\eta\right)
W_{1}\nabla\phi\right\vert \leq\left\vert \int\left(  \eta^{m}-\eta\right)
\left(  S_{q}W_{1}\right)  \nabla\phi\right\vert +\left\vert \int\left(
\eta^{m}-\eta\right)  \left(  \left(  Id-S_{q}\right)  W_{1}\right)
\nabla\phi\right\vert \nonumber\\
&  \leq\left\vert \int\left(  \eta^{m}-\eta\right)  \left(  S_{q}W_{1}\right)
\nabla\phi\right\vert +C\left\Vert \left(  Id-S_{q}\right)  W_{1}\right\Vert
_{B_{p_{1},r}^{s+1}}\left\Vert \left(  \eta^{m}-\eta\right)  \nabla
\phi\right\Vert _{B_{p_{1}^{\prime},r^{\prime}}^{-s-1}}\nonumber\\
&  \leq\left\vert \int\left(  \eta^{m}-\eta\right)  \left(  S_{q}W_{1}\right)
\nabla\phi\right\vert +C\left\Vert \left(  Id-S_{q}\right)  \nabla
W_{1}\right\Vert _{B_{p_{1},r}^{s}}\left\Vert \left(  \eta^{m}-\eta\right)
\nabla\phi\right\Vert _{B_{p_{1}^{\prime},\infty}^{0}}\nonumber\\
&  \leq\left\vert \int\left(  \eta^{m}-\eta\right)  \left(  S_{q}W_{1}\right)
\nabla\phi\right\vert +C\left\Vert \left(  Id-S_{q}\right)  \nabla
W_{1}\right\Vert _{B_{p_{1},r}^{s}}\left(  \left\Vert \eta^{m}\right\Vert
_{L^{p_{2}}}+\left\Vert \eta\right\Vert _{L^{p_{2}}}\right)  \left\Vert
\nabla\phi\right\Vert _{L^{2}}\nonumber\\
&  \leq\left\vert \int\left(  \eta^{m}-\eta\right)  \left(  S_{q}W_{1}\right)
\nabla\phi\right\vert +CU_{s}\left(  0\right)  \left\Vert \left(
Id-S_{q}\right)  \nabla W_{1}\right\Vert _{B_{p_{1},r}^{s}}\left\Vert
\nabla\phi\right\Vert _{L^{2}}.\label{convtermen1}%
\end{align}
The fact that the first term of $\left(  \text{\ref{convtermen1}}\right)  $
tends to zero as $m\rightarrow\infty$ is a consequence of $\left(
\text{\ref{conveta}}\right)  $. The second term tends to zero as
$q\rightarrow\infty$ owing to the fact that $\nabla W_{1}\in B_{p_{1},r}^{s}$.
Let us also show how to deal with the nonlinear terms. We write that%
\begin{align}
\left\vert \int\phi\operatorname{div}\left(  \eta^{m}V^{m}-\eta V\right)
\right\vert  &  =\left\vert \int\left(  \eta^{m}V^{m}-\eta V\right)
\nabla\phi\right\vert \leq\left\vert \int\eta^{m}\left(  V^{m}-V\right)
\nabla\phi\right\vert +\left\vert \int\left(  \eta^{m}-\eta\right)
V\nabla\phi\right\vert \nonumber\\
&  \leq C\left\Vert V^{m}\nabla\phi-V\nabla\phi\right\Vert _{B_{2,r}^{s-1}%
}\left\Vert \eta^{m}\right\Vert _{B_{2,r^{\prime}}^{1-s}}+\left\vert
\int\left(  \eta^{m}-\eta\right)  V\nabla\phi\right\vert \nonumber\\
&  \leq C\left\Vert V^{m}\nabla\phi-V\nabla\phi\right\Vert _{B_{2,r}^{s-1}%
}\left\Vert \eta^{m}\right\Vert _{B_{2,r}^{s}}+\left\vert \int\left(  \eta
^{m}-\eta\right)  V\nabla\phi\right\vert \nonumber\\
&  \leq CU_{s}\left(  0\right)  \left\Vert V^{m}\nabla\phi-V\nabla
\phi\right\Vert _{B_{2,r}^{s-1}}+\left\vert \int\left(  \eta^{m}-\eta\right)
V\nabla\phi\right\vert .\label{convtermen2}%
\end{align}
The first term of $\left(  \text{\ref{convtermen2}}\right)  $ tends to zero as
$m\rightarrow\infty$, owing to relation $\left(  \text{\ref{convv}}\right)  $.
In order to show that the second term of $\left(  \text{\ref{convtermen2}%
}\right)  $ tends to zero as $m\rightarrow\infty$ one proceeds exactly like we
did in $\left(  \text{\ref{convtermen1}}\right)  $.

Recovering the time regularity of $\left(  \eta,V\right)  $ is again
classical. One can show for example that for all $j\in\mathbb{Z}$ we have%
\[
S_{j}\eta\in\text{ }\mathcal{C}\left(  [0,T],B_{2,r}^{s_{1}}\right)
\]
and by using energy estimates:
\[
\lim_{j\rightarrow\infty}\left\Vert \eta-S_{j}\eta\right\Vert _{L_{\bar{T}%
}^{\infty}\left(  B_{2,r}^{s_{1}}\right)  }=0.
\]
As for the uniqueness of solutions let us consider $\left(  \eta^{1}%
,V^{1}\right)  $, $\left(  \eta^{2},V^{2}\right)  $ two solutions of $\left(
\text{\ref{BBM}}\right)  $. The system verified by
\[
\left(  \delta\eta,\delta V\right)  =\left(  \eta^{1}-\eta^{2},V^{1}%
-V^{2}\right)
\]
is the following:%
\begin{equation}
\left\{
\begin{array}
[c]{l}%
\left(  I-\varepsilon b\Delta\right)  \partial_{t}\delta\eta
+\operatorname{div}\delta V+\varepsilon\operatorname{div}\left(  \delta
\eta\tilde{W}_{1}+\tilde{h}\delta V\right)  =0,\\
\left(  I-\varepsilon d\Delta\right)  \partial_{t}\delta V+\nabla\delta
\eta+\varepsilon\tilde{W}_{2}\cdot\nabla\delta V+\delta V\cdot\nabla\tilde
{W}_{3}=0\\
\eta_{|t=0}=0,\text{ }V_{|t=0}=0,
\end{array}
\right.  \label{unicitate1}%
\end{equation}
with%
\[
\left\{
\begin{array}
[c]{cc}%
\tilde{h}=h+\beta\eta^{2}, & \tilde{W}_{1}=W_{1}+\beta V^{1},\\
\tilde{W}_{2}=W_{2}+\beta V^{1}, & \tilde{W}_{3}=W_{3}+\beta V^{2}.
\end{array}
\right.
\]
Let us multiply the firs equation of $\left(  \text{\ref{unicitate1}}\right)
$ with $\delta\eta$ and the second one with $\delta V$ such that by repeated
integration by parts, one obtains:%
\begin{align}
\frac{1}{2}\frac{d}{dt}\delta U^{2}\left(  t\right)   &  \leq C\left(
\varepsilon\left\Vert \operatorname{div}\tilde{W}_{1}\right\Vert _{L^{\infty}%
}+\varepsilon\left\Vert \operatorname{div}\tilde{W}_{2}\right\Vert
_{L^{\infty}}\right)  \delta U^{2}\left(  t\right) \label{unicitate2}\\
&  +C\left(  \varepsilon\left\Vert \nabla\tilde{W}_{3}\right\Vert _{L^{\infty
}}+\frac{\sqrt{\varepsilon}\left(  \left\Vert \tilde{h}\right\Vert
_{L^{\infty}}+\left\Vert \nabla\tilde{h}\right\Vert _{L^{\infty}}\right)
}{\max\left\{  \sqrt{b},\sqrt{d}\right\}  }\right)  \delta U^{2}\left(
t\right) \nonumber
\end{align}
with%
\[
\delta U^{2}\left(  t\right)  :\overset{not.}{=}\int\left\Vert \left(
\delta\eta,\delta V\right)  \right\Vert _{L^{2}}^{2}+\left\Vert \left(
b\nabla\delta\eta,d\nabla\delta V\right)  \right\Vert _{L^{2}}^{2}.
\]
As $\delta U\left(  0\right)  =0$, uniqueness follows by Gronwall's Lemma and
thus the proof of Theorem \ref{Teorema1} is achieved.

\subsection{The lower bound on the time of existence}

In this section we prove Theorem \ref{Teorema1.1}. Let us consider $\left(
\eta_{0},V_{0}\right)  \in$ $X_{b,d,r}^{s}$ and let us denote by
\[
R_{0}^{\varepsilon}\overset{not.}{=}\left\Vert \left(  \eta_{0},V_{0}\right)
\right\Vert _{X_{b,d,r}^{s,\varepsilon}}.
\]
Then, according to Theorem \ref{Teorema1}, for all $\varepsilon>0$, there
exists an unique maximal solution of $\mathcal{S}_{\varepsilon}\left(
\mathcal{D}^{\varepsilon}\right)  $ which we denote by $\left(  \eta
^{\varepsilon}\left(  t\right)  ,V^{\varepsilon}\left(  t\right)  \right)  \in
X_{b,d,r}^{s}$. We introduce the following notations:%
\begin{equation}
\left\{
\begin{array}
[c]{l}%
U_{j}^{2}(\varepsilon,t)=\left\Vert \left(  \eta_{j}^{\varepsilon}%
(t),V_{j}^{\varepsilon}\left(  t\right)  \right)  \right\Vert _{L^{2}}%
^{2}+\varepsilon\left\Vert \left(  \sqrt{b}\nabla\eta_{j}^{\varepsilon}\left(
t\right)  ,\sqrt{d}\nabla V_{j}^{\varepsilon}\left(  t\right)  \right)
\right\Vert _{L^{2}}^{2},\\
\left\Vert \left(  \eta^{\varepsilon}\left(  t\right)  ,V^{\varepsilon}\left(
t\right)  \right)  \right\Vert _{X_{b,d,r}^{s,\varepsilon}}=U_{s}\left(
\varepsilon,t\right)  =\left\Vert \left(  2^{js}U_{j}\left(  \varepsilon
,t\right)  \right)  _{j\in\mathbb{Z}}\right\Vert _{\ell^{r}(\mathbb{Z})}.
\end{array}
\right.  \label{NotEps}%
\end{equation}
Let us consider
\[
T_{\star}^{\varepsilon}=\sup\left\{  T\in\lbrack0,T^{\varepsilon}]:\sup
_{t\in\left[  0,T\right]  }\left\Vert \left(  \eta^{\varepsilon}\left(
t\right)  ,V^{\varepsilon}\left(  t\right)  \right)  \right\Vert
_{X_{b,d,r}^{s,\varepsilon}}\leq\left(  1+e\sqrt{7}\right)  R_{0}%
^{\varepsilon}\right\}  ,
\]
and
\[
\varepsilon_{0}=\min\left\{  \varepsilon_{01},\varepsilon_{02},\varepsilon
_{03},\varepsilon_{04}\right\}  ,
\]
where%
\[
\left\{
\begin{array}
[c]{l}%
\varepsilon_{01}=\frac{3}{4C_{1}\left(  1+e\sqrt{7}\right)  R_{0}^{1}%
+4\sup\limits_{\tilde{\varepsilon}\in\left[  0,1\right]  }\sup\limits_{\tau
\in\left[  0,T^{\tilde{\varepsilon}}\right]  }\left\Vert h^{\tilde
{\varepsilon}}\left(  \tau\right)  \right\Vert _{L^{\infty}}},\text{
}\varepsilon_{02}=\frac{1}{2\left(  1+e\sqrt{7}\right)  R_{0}^{1}%
+2\sup\limits_{\tilde{\varepsilon}\in\left[  0,1\right]  }\sup\limits_{\tau
\in\left[  0,T^{\tilde{\varepsilon}}\right]  }\mathcal{W}_{s}^{\tilde
{\varepsilon}}\left(  \tau\right)  },\\
\varepsilon_{03}=\frac{\left(  1+e\sqrt{7}\right)  R_{0}^{0}+\sup
\limits_{\tilde{\varepsilon}\in\left[  0,1\right]  }\sup\limits_{\tau
\in\left[  0,T^{\tilde{\varepsilon}}\right]  }\mathcal{W}_{s}^{\tilde
{\varepsilon}}\left(  \tau\right)  }{2\sup\limits_{\tilde{\varepsilon}%
\in\left[  0,1\right]  }\sup\limits_{\tau\in\left[  0,T^{\tilde{\varepsilon}%
}\right]  }F_{s}^{\tilde{\varepsilon}}\left(  \tau\right)  },\text{
}\varepsilon_{04}=\frac{1}{2\left(  1+e\sqrt{7}\right)  R_{0}^{1}%
+2\sup\limits_{\tilde{\varepsilon}\in\left[  0,1\right]  }\sup\limits_{\tau
\in\left[  0,T^{\tilde{\varepsilon}}\right]  }\mathcal{W}_{s}^{\tilde
{\varepsilon}}\left(  \tau\right)  }.
\end{array}
\right.
\]
where $C_{1}$ is the constant appearing in the embedding $B_{2,1}^{\frac{n}%
{2}}\hookrightarrow L^{\infty}$ i.e.%
\[
\left\Vert f\right\Vert _{L^{\infty}}\leq C_{1}\left\Vert f\right\Vert
_{B_{2,1}^{\frac{n}{2}}}.
\]
Let us put%
\begin{equation}
\tilde{C}=\min\left\{  \frac{R_{0}^{0}}{3eC\sup\limits_{\tilde{\varepsilon}%
\in\left[  0,1\right]  }\sup\limits_{\tau\in\left[  0,T^{\tilde{\varepsilon}%
}\right]  }F_{s}^{\tilde{\varepsilon}}\left(  \tau\right)  },\frac
{1}{16C\left(  1+e\sqrt{7}\right)  R_{0}^{1}},\frac{1}{16C\sup\limits_{\tilde
{\varepsilon}\in\left[  0,1\right]  }\sup\limits_{\tau\in\left[
0,T^{\tilde{\varepsilon}}\right]  }\mathcal{W}_{s}^{\tilde{\varepsilon}%
}\left(  \tau\right)  }\right\}  \label{constanta1}%
\end{equation}
where $C$ is the universal constant appearing in $\left(  \text{\ref{final1}%
}\right)  $. We claim that for all $\varepsilon\leq\varepsilon_{0}$
\[
T_{\star}^{\varepsilon}\geq\min\left\{  \frac{\tilde{C}}{\varepsilon
},T^{\varepsilon}\right\}  .
\]
Let us suppose that this is not the case. Then there exists an $\varepsilon>0$
such that
\begin{equation}
T_{\star}^{\varepsilon}<\min\left\{  \frac{\tilde{C}}{\varepsilon
},T^{\varepsilon}\right\}  .\label{bound1}%
\end{equation}
We begin by writing that%
\[
\frac{1}{4}\leq1+\varepsilon\left(  \eta^{\varepsilon}\left(  t,x\right)
+h^{\varepsilon}\left(  t,x\right)  \right)  \leq\frac{7}{4},
\]
for all $\left(  t,x\right)  \in\left[  0,T_{\star}^{\varepsilon}\right]
\times\mathbb{R}^{n}$, owing to the fact that
\begin{equation}
\varepsilon\leq\varepsilon_{01}.\label{eps1}%
\end{equation}
Thus, denoting by
\[
N_{j}^{2}\left(  \varepsilon,t\right)  :=\int_{\mathbb{R}^{n}}\left(  \eta
_{j}^{\varepsilon}\right)  ^{2}\left(  t\right)  +\varepsilon b\left\vert
\nabla\eta_{j}^{\varepsilon}\left(  t\right)  \right\vert ^{2}+\left(
1+\varepsilon\left(  \eta^{\varepsilon}\left(  t\right)  +h^{\varepsilon
}\left(  t\right)  \right)  \right)  \left(  \left\vert V_{j}^{\varepsilon
}\right\vert ^{2}\left(  t\right)  +\varepsilon d\nabla V_{j}^{\varepsilon
}\left(  t\right)  :\nabla V_{j}^{\varepsilon}\left(  t\right)  \right)
\]
we see that for all $t\in\left[  0,T_{\star}^{\varepsilon}\right]  $ we get
that:%
\begin{equation}
\frac{1}{2}U_{j}\left(  \varepsilon,t\right)  \leq N_{j}\left(  \varepsilon
,t\right)  \leq\frac{\sqrt{7}}{2}U_{j}\left(  \varepsilon,t\right)
.\label{ineg}%
\end{equation}
Recall that according to $\left(  \text{\ref{final1}}\right)  $ we have that:%
\begin{gather}
\frac{d}{dt}N_{j}^{2}\left(  \varepsilon,t\right)  \leq\varepsilon CU_{j}%
^{2}\left(  \varepsilon,t\right)  \left(  \varepsilon F_{s}^{\varepsilon
}\left(  t\right)  +\mathcal{W}_{s}^{\varepsilon}\left(  t\right)
+U_{s}\left(  \varepsilon,t\right)  +\varepsilon\left(  \mathcal{W}%
_{s}^{\varepsilon}\left(  t\right)  +U_{s}\left(  \varepsilon,t\right)
\right)  ^{2}\right)  \label{final1eps}\\
+\varepsilon C2^{-js}c_{j}^{\varepsilon}\left(  t\right)  U_{j}\left(
\varepsilon,t\right)  F_{s}^{\varepsilon}\left(  t\right)  \left\{
1+\varepsilon\left(  \mathcal{W}_{s}^{\varepsilon}\left(  t\right)
+U_{s}\left(  \varepsilon,t\right)  \right)  \right\}  \\
+\varepsilon C2^{-js}c_{j}^{\varepsilon}\left(  t\right)  U_{j}\left(
\varepsilon,t\right)  U_{s}\left(  \varepsilon,t\right)  \left(
\mathcal{W}_{s}^{\varepsilon}\left(  t\right)  +U_{s}\left(  \varepsilon
,t\right)  \right)  \left\{  1+\varepsilon\left(  \mathcal{W}_{s}%
^{\varepsilon}\left(  t\right)  +U_{s}\left(  \varepsilon,t\right)  \right)
\right\}
\end{gather}
and using $\left(  \text{\ref{ineg}}\right)  $ we get that:%
\begin{gather}
U_{j}\left(  \varepsilon,t\right)  \leq\sqrt{7}U_{j}\left(  \varepsilon
,0\right)  +2\varepsilon C\int_{0}^{t}U_{j}\left(  \varepsilon,\tau\right)
\left(  \varepsilon F_{s}^{\varepsilon}\left(  \tau\right)  +\mathcal{W}%
_{s}^{\varepsilon}\left(  \tau\right)  +U_{s}\left(  \varepsilon,\tau\right)
+\varepsilon\left(  \mathcal{W}_{s}^{\varepsilon}\left(  \tau\right)
+U_{s}\left(  \varepsilon,\tau\right)  \right)  ^{2}\right)  \nonumber\\
+2\varepsilon C2^{-js}\int_{0}^{t}c_{j}\left(  F_{s}^{\varepsilon}\left(
\tau\right)  \left(  1+\varepsilon\left(  \mathcal{W}_{s}^{\varepsilon}\left(
\tau\right)  +U_{s}\left(  \varepsilon,\tau\right)  \right)  \right)
+U_{s}\left(  \varepsilon,\tau\right)  \left(  \mathcal{W}_{s}^{\varepsilon
}\left(  \tau\right)  +U_{s}\left(  \varepsilon,\tau\right)  \right)  \left(
1+\varepsilon\left(  \mathcal{W}_{s}^{\varepsilon}\left(  \tau\right)
+U_{s}\left(  \varepsilon,\tau\right)  \right)  \right)  \right)
.\label{ineg2}%
\end{gather}
Multiplying $\left(  \text{\ref{ineg2}}\right)  $ with $2^{js}$ and performing
an $\ell^{r}\left(  \mathbb{Z}\right)  $-summation we end up with%
\begin{align}
U_{s}\left(  \varepsilon,t\right)   &  \leq\sqrt{7}U_{s}\left(  \varepsilon
,0\right)  +2\varepsilon C\int_{0}^{t}F_{s}^{\varepsilon}\left(  \tau\right)
\left(  1+\varepsilon\left(  \mathcal{W}_{s}^{\varepsilon}\left(  \tau\right)
+U_{s}\left(  \varepsilon,\tau\right)  \right)  \right)  d\tau\nonumber\\
&  \text{ \ \ \ \ \ \ \ \ \ \ \ \ \ \ }+4\varepsilon C\int_{0}^{t}U_{s}\left(
\varepsilon,\tau\right)  \left(  \varepsilon F_{s}^{\varepsilon}\left(
\tau\right)  +\mathcal{W}_{s}^{\varepsilon}\left(  \tau\right)  +U_{s}\left(
\varepsilon,\tau\right)  +\varepsilon\left(  \mathcal{W}_{s}^{\varepsilon
}\left(  \tau\right)  +U_{s}\left(  \varepsilon,\tau\right)  \right)
^{2}\right)  d\tau\nonumber\\
&  \leq\sqrt{7}R_{0}^{\varepsilon}+3\varepsilon CT_{\star}^{\varepsilon}%
\sup_{\tilde{\varepsilon}\in\left[  0,1\right]  }\sup_{\tau\in\left[
0,T^{\tilde{\varepsilon}}\right]  }F_{s}^{\tilde{\varepsilon}}\left(
\tau\right)  \nonumber\\
&  \text{ \ \ \ \ \ \ \ \ \ \ \ \ \ \ }+4\varepsilon C\int_{0}^{t}U_{s}\left(
\varepsilon,\tau\right)  \left(  \varepsilon F_{s}^{\varepsilon}\left(
\tau\right)  +\mathcal{W}_{s}^{\varepsilon}\left(  \tau\right)  +U_{s}\left(
\varepsilon,\tau\right)  +\varepsilon\left(  \mathcal{W}_{s}^{\varepsilon
}+U_{s}\left(  \varepsilon,\tau\right)  \right)  ^{2}\right)  d\tau
,\label{estimare1}%
\end{align}
owing to the fact that $\varepsilon$ has been chosen such that:%
\begin{equation}
\varepsilon\leq\varepsilon_{02}.\label{eps2}%
\end{equation}
Using the definition of $\tilde{C}$ we get that
\begin{equation}
T_{\star}^{\varepsilon}<\frac{R_{0}^{0}}{3\varepsilon eC\sup\limits_{\tilde
{\varepsilon}\in\left[  0,1\right]  }\sup\limits_{\tau\in\left[
0,T^{\tilde{\varepsilon}}\right]  }F_{s}^{\tilde{\varepsilon}}\left(
\tau\right)  }<\frac{R_{0}^{\varepsilon}}{3\varepsilon eC\sup\limits_{\tilde
{\varepsilon}\in\left[  0,1\right]  }\sup\limits_{\tau\in\left[
0,T^{\tilde{\varepsilon}}\right]  }F_{s}^{\tilde{\varepsilon}}\left(
\tau\right)  }\label{bound2}%
\end{equation}
and consequently, we get that:%
\begin{equation}
U_{s}\left(  \varepsilon,t\right)  \leq\left(  e^{-1}+\sqrt{7}\right)
R_{0}^{\varepsilon}+8\varepsilon C\left(  \left(  1+e\sqrt{7}\right)
R_{0}^{\varepsilon}+\sup\limits_{\tilde{\varepsilon}\in\left[  0,1\right]
}\sup\limits_{\tau\in\left[  0,T^{\tilde{\varepsilon}}\right]  }%
\mathcal{W}_{s}^{\tilde{\varepsilon}}\left(  \tau\right)  \right)  \int
_{0}^{t}U_{s}\left(  \varepsilon,\tau\right)  d\tau,\label{estimare2}%
\end{equation}
where we have used that $\varepsilon\leq\varepsilon_{03}$ respectively
$\varepsilon\leq\varepsilon_{04}$. The estimate $\left(  \text{\ref{estimare2}%
}\right)  $ along with Gronwall's Lemma imply that:%
\[
U_{s}\left(  \varepsilon,t\right)  \leq\left(  e^{-1}+\sqrt{7}\right)
R_{0}^{\varepsilon}\exp\left(  8\varepsilon CT_{\star}^{\varepsilon}\left(
\left(  1+e\sqrt{7}\right)  R_{0}^{\varepsilon}+\sup_{\tilde{\varepsilon}%
\in\left[  0,1\right]  }\sup_{\tau\in\left[  0,T^{\tilde{\varepsilon}}\right]
}\mathcal{W}_{s}^{\tilde{\varepsilon}}\left(  \tau\right)  \right)  \right)  .
\]
Owing to
\begin{align*}
T_{\star}^{\varepsilon} &  <\frac{1}{\varepsilon}\min\left\{  \frac
{1}{16C\left(  1+e\sqrt{7}\right)  R_{0}^{1}},\frac{1}{16C\sup\limits_{\tilde
{\varepsilon}\in\left[  0,1\right]  }\sup\limits_{\tau\in\left[
0,T^{\tilde{\varepsilon}}\right]  }\mathcal{W}_{s}^{\tilde{\varepsilon}%
}\left(  \tau\right)  }\right\}  \\
&  \leq\frac{1}{8\varepsilon C\left(  \left(  1+e\sqrt{7}\right)  R_{0}%
^{1}+\sup\limits_{\tilde{\varepsilon}\in\left[  0,1\right]  }\sup
\limits_{\tau\in\left[  0,T^{\tilde{\varepsilon}}\right]  }\mathcal{W}%
_{s}^{\tilde{\varepsilon}}\left(  \tau\right)  \right)  }%
\end{align*}
we obtain that
\[
\sup_{t\in\left[  0,T_{\star}^{\varepsilon}\right]  }U_{s}\left(
\varepsilon,t\right)  <\left(  1+e\sqrt{7}\right)  R_{0}^{\varepsilon}%
\]
which in view of the time continuity of $\left(  \eta^{\varepsilon
},V^{\varepsilon}\right)  $ is a contradiction. Thus, our initial suppositions
is false. Thus, if $\varepsilon\leq\varepsilon_{0}$ we get that
\[
T_{\star}^{\varepsilon}\geq\min\left\{  T^{\varepsilon},\frac{\tilde{C}%
}{\varepsilon}\right\}  .
\]
By the definition of $T_{\star}^{\varepsilon}$ we have that
\begin{equation}
\sup_{t\in\left[  0,T_{\star}^{\varepsilon}\right]  }\left\Vert \left(
\eta^{\varepsilon}\left(  t\right)  ,V^{\varepsilon}\left(  t\right)  \right)
\right\Vert _{X_{b,d,r}^{s,\varepsilon}}\leq\left(  1+e\sqrt{7}\right)
R_{0}^{\varepsilon}.\label{ineg3}%
\end{equation}
Observe that according to the uniform bounds of $\left(  \text{\ref{ineg3}%
}\right)  $ and relation $\left(  \text{\ref{derivatatemporala}}\right)  $ we
get that%
\[
\sup_{t\in\left[  0,T_{\star}^{\varepsilon}\right]  }\left\Vert \partial
_{t}\eta^{\varepsilon}\left(  t\right)  \right\Vert _{L^{\infty}}%
\leq2C^{^{\prime}}\left(  1+e\sqrt{7}\right)  R_{0}^{\varepsilon},
\]
where $C^{\prime}$ is the constant appearing in relation $\left(
\text{\ref{derivatatemporala}}\right)  $. This ends the proof of Theorem
\ref{Teorema1.1}.

\section{Proofs of the main results}

\subsection{The 1d case: proof of Theorem \ref{Teorema2}}

In the following lines we prove the 1-dimensional result announced inTheorem
\ref{Teorema2}. For the reader's convenience, let us write below the system:%

\begin{equation}
\left\{
\begin{array}
[c]{l}%
\left(  I-\varepsilon b\partial_{xx}^{2}\right)  \partial_{t}\bar{\eta
}+\partial_{x}\bar{u}+\varepsilon\partial_{x}\left(  \bar{\eta}\bar{u}\right)
=0,\\
\left(  I-\varepsilon d\partial_{xx}^{2}\right)  \partial_{t}\bar{u}%
+\partial_{x}\bar{\eta}+\varepsilon\bar{u}\partial_{x}\bar{u}=0,\\
\bar{\eta}_{|t=0}=\eta_{0},\text{ }\bar{u}_{|t=0}=u_{0}.
\end{array}
\right.  \label{BBM1d}%
\end{equation}
The strategy of the proof is the following. First, we split the initial data
into low-high frequency parts i.e. :%
\begin{align*}
\left(  \eta_{0},u_{0}\right)   &  =\left(  \Delta_{-1}\eta_{0},\Delta
_{-1}u_{0}\right)  +\left(  \left(  I-\Delta_{-1}\right)  \eta_{0},\left(
I-\Delta_{-1}\right)  u_{0}\right) \\
&  \overset{not.}{=}\left(  \eta_{0}^{low},u_{0}^{low}\right)  +\left(
\eta_{0}^{high},u_{0}^{high}\right)  .
\end{align*}
We solve the linear acoustic waves system with initial data $\left(  \eta
_{0}^{low},u_{0}^{low}\right)  $. Searching for a solution $\left(  \bar{\eta
},\bar{u}\right)  $ of $\left(  \text{\ref{BBM1d}}\right)  $ in the form
$\left(  \eta+\eta_{L},u+u_{L}\right)  $ we observe that in fact $\left(
\eta,u\right)  $ verifies a system of type $\left(  \text{\ref{BBM}}\right)
$. In view of the uniform bounds (with respect to time) of $\left(  \eta
_{L},u_{L}\right)  $ we obtain the existence of the pair $\left(
\eta,u\right)  $ as a consequence of Corollary \ref{Corolar}. Finally, we
prove the uniqueness property by performing classical $L^{2}$-energy estimates
with respect to the system of equations governing the difference of two
solutions with the same initial data.

Owing to Bernstein's Lemma $\left(  \text{\ref{Bernstein}}\right)  $ we have:%
\begin{equation}
\left\Vert \left(  \eta_{0}^{low},u_{0}^{low}\right)  \right\Vert _{L^{\infty
}}\leq C_{1}\left\Vert \left(  \eta_{0},u_{0}\right)  \right\Vert _{L^{\infty
}}.\label{low_1d}%
\end{equation}
Moreover, it transpires that $\left(  \eta_{0}^{high},u_{0}^{high}\right)  \in
X_{b,d,r}^{s}$ and that%
\begin{equation}
\left\Vert \left(  \eta_{0}^{high},u_{0}^{high}\right)  \right\Vert
_{X_{b,d,r}^{s,\varepsilon}}\leq C_{2}\left\Vert \left(  \partial_{x}\eta
_{0},\partial_{x}u_{0}\right)  \right\Vert _{X_{b,d,r}^{s-1,\varepsilon}%
}.\label{high_1d}%
\end{equation}
Let us consider $\left(  \eta_{L},u_{L}\right)  \in\mathcal{C}\left(
[0,\infty),E_{b,d,r}^{s}\left(  \mathbb{R}\right)  \right)  $ the unique
solution of the linear acoustic waves system:
\begin{equation}
\left\{
\begin{array}
[c]{l}%
\partial_{t}\eta_{L}+\partial_{x}u_{L}=0,\\
\partial_{t}u_{L}+\partial_{x}\eta_{L}=0,\\
\eta_{L|t=0}=\eta_{0}^{low}\text{ },\text{ }u_{L|t=0}=u_{0}^{low}%
\end{array}
\right.  \label{sistem_unde1d}%
\end{equation}
which is given explicitly by the following relation:
\begin{equation}
\left\{
\begin{array}
[c]{c}%
2\eta_{L}\left(  t,x\right)  =\left(  \eta_{0}^{low}\left(  x+t\right)
+\eta_{0}^{low}\left(  x-t\right)  \right)  +\left(  u_{0}^{low}\left(
x+t\right)  -u_{0}^{low}\left(  x-t\right)  \right)  ,\\
2u_{L}\left(  t,x\right)  =\left(  \eta_{0}^{low}\left(  x+t\right)  -\eta
_{0}^{low}\left(  x-t\right)  \right)  +\left(  u_{0}^{low}\left(  x+t\right)
+u_{0}^{low}\left(  x-t\right)  \right)  ,
\end{array}
\right.  \label{unde_explicit}%
\end{equation}
for all $\left(  t,x\right)  \in\mathbb{R}^{+}\mathbb{\times R}$. From
$\left(  \text{\ref{unde_explicit}}\right)  $ we conclude that:%
\begin{align}
\left\Vert \left(  \eta_{L}\left(  t\right)  ,u_{L}\left(  t\right)  \right)
\right\Vert _{L^{\infty}} &  \leq2C_{1}\left\Vert \left(  \eta_{0}%
,u_{0}\right)  \right\Vert _{L^{\infty}},\nonumber\\
\left\Vert \left(  \partial_{x}\eta_{L}\left(  t\right)  ,\partial_{x}%
u_{L}\left(  t\right)  \right)  \right\Vert _{B_{2,r}^{\sigma}} &  \leq
2C_{1}\left\Vert \left(  \partial_{x}\eta_{0},\partial_{x}u_{0}\right)
\right\Vert _{B_{2,\infty}^{0}}\label{conc2}%
\end{align}
for all $\sigma\geq0$. In particular, we also have%
\begin{equation}
\left\Vert \left(  \eta_{L}\left(  t\right)  ,u_{L}\left(  t\right)  \right)
\right\Vert _{E_{b,d,r}^{s,\varepsilon}}\leq2C_{1}\left\Vert \left(  \eta
_{0},u_{0}\right)  \right\Vert _{E_{b,d}^{s,\varepsilon}}\label{normaliniara}%
\end{equation}
for all $t\geq0$. Let us consider$:$%
\[
\left\{
\begin{array}
[c]{l}%
f_{L}=-\eta_{L}\partial_{x}u_{L}-u_{L}\partial_{x}\eta_{L}-b\partial_{xxx}%
^{3}u_{L},\\
g_{L}=-u_{L}\partial_{x}u_{L}-d\partial_{xxx}^{3}\eta_{L}.
\end{array}
\right.
\]
Owing to the fact that $\left(  \widehat{\eta_{L}},\widehat{u_{L}}\right)  $
is supported in a ball centered at the origin, we obtain that
\begin{align}
\left\Vert \left(  f_{L}\left(  t\right)  ,g_{L}\left(  t\right)  \right)
\right\Vert _{B_{2,r}^{s}} &  \leq C_{s}\left\{  \left\Vert \left(  \eta
_{L},u_{L}\right)  \right\Vert _{L^{\infty}}\left\Vert \left(  \partial
_{x}\eta_{L},\partial_{x}u_{L}\right)  \right\Vert _{B_{2,r}^{s}}+\left\Vert
\left(  \partial_{x}\eta_{L},\partial_{x}u_{L}\right)  \right\Vert
_{L^{\infty}}\left\Vert \left(  \partial_{x}\eta_{L},\partial_{x}u_{L}\right)
\right\Vert _{B_{2,r}^{s-1}}\right.  \nonumber\\
&  \text{ \ \ \ \ \ }\left.  \text{\ }+b\left\Vert \partial_{xxx}^{3}%
u_{L}\right\Vert _{B_{2,r}^{s}}+d\left\Vert \partial_{xxx}^{3}\eta
_{L}\right\Vert _{B_{2,r}^{s}}\right\}  \nonumber\\
&  \leq C_{s}\left\Vert \left(  \partial_{x}\eta_{0},\partial_{x}u_{0}\right)
\right\Vert _{B_{2,\infty}^{0}}\left(  \left\Vert \left(  \eta_{0}%
,u_{0}\right)  \right\Vert _{L^{\infty}}+\left\Vert \left(  \partial_{x}%
\eta_{0},\partial_{x}u_{0}\right)  \right\Vert _{B_{2,\infty}^{0}}+b+d\right)
.\label{lin1.1}%
\end{align}
Owing to the uniform bounds of $\left(  \eta_{L},u_{L},f_{L},g_{L}\right)  $
announced in $\left(  \text{\ref{conc2}}\right)  $ and in $\left(
\text{\ref{lin1.1}}\right)  $ we can apply Corollary \ref{Corolar} with
$p_{1}=2$, $p_{2}=\infty$ in order to obtain the existence of two positive
reals $\varepsilon_{0},C$ depending on $s,b,d$ and the norm of the initial
data such that for any $\varepsilon\leq\varepsilon_{0}$ we may consider
$\left(  \eta^{\varepsilon},u^{\varepsilon}\right)  \in\mathcal{C}\left(
[0,\frac{C}{\varepsilon}],X_{b,d,r}^{s}\right)  $ the solution of%
\begin{equation}
\left\{
\begin{array}
[c]{l}%
\left(  I-\varepsilon b\partial_{xx}^{2}\right)  \partial_{t}\eta+\partial
_{x}u+\varepsilon\partial_{x}\left(  \eta u_{L}+\eta_{L}u+\eta u\right)
=\varepsilon f_{L},\\
\left(  I-\varepsilon d\partial_{xx}^{2}\right)  \partial_{t}u+\partial
_{x}\eta+\varepsilon\left(  u+u_{L}\right)  \partial_{x}u+u\partial_{x}%
u_{L}=\varepsilon g_{L},\\
\eta_{|t=0}=\eta_{0}^{high},\text{ }u_{|t=0}=u_{0}^{high}%
\end{array}
\right.  \label{BBM2}%
\end{equation}
which satisfies
\begin{align}
\sup_{t\in\left[  0,\frac{C}{\varepsilon}\right]  }\left(  \left\Vert \left(
\eta^{\varepsilon}\left(  t\right)  ,u^{\varepsilon}\left(  t\right)  \right)
\right\Vert _{X_{b,d,r}^{s,\varepsilon}}+\left\Vert \partial_{t}%
\eta^{\varepsilon}\left(  t\right)  \right\Vert _{L^{\infty}}\right)   &
\leq\tilde{C}\left\Vert \left(  \eta_{0}^{high},u_{0}^{high}\right)
\right\Vert _{X_{b,d,r}^{s,\varepsilon}}\nonumber\\
&  \leq\tilde{C}C_{2}\left\Vert \left(  \partial_{x}\eta_{0},\partial_{x}%
u_{0}\right)  \right\Vert _{X_{b,d,r}^{s-1,\varepsilon}},\label{conc1}%
\end{align}
for some numerical constant $\tilde{C}$. Considering
\[
\bar{\eta}^{\varepsilon}=\eta^{\varepsilon}+\eta_{L},\text{ }\bar
{u}^{\varepsilon}=u^{\varepsilon}+u_{L}%
\]
we see that $\left(  \bar{\eta}^{\varepsilon},\bar{u}^{\varepsilon}\right)
\in$ $\mathcal{C}\left(  [0,\infty),E_{b,d,r}^{s}\left(  \mathbb{R}\right)
\right)  $ and for all $t\in\lbrack0,\frac{C}{\varepsilon}]$ we have that:
\begin{align*}
&  \left\Vert \left(  \bar{\eta}^{\varepsilon}\left(  t\right)  ,\bar
{u}\left(  t\right)  \right)  \right\Vert _{E_{b,d}^{s,\varepsilon}%
}+\left\Vert \partial_{t}\bar{\eta}^{\varepsilon}\left(  t\right)  \right\Vert
_{L^{\infty}}\\
&  \leq\left\Vert \left(  \eta^{\varepsilon}\left(  t\right)  ,u^{\varepsilon
}\left(  t\right)  \right)  \right\Vert _{E_{b,d,r}^{s,\varepsilon}%
}+\left\Vert \partial_{t}\eta^{\varepsilon}\left(  t\right)  \right\Vert
_{L^{\infty}}+\left\Vert \left(  \eta_{L}\left(  t\right)  ,u_{L}\left(
t\right)  \right)  \right\Vert _{E_{b,d,r}^{s,\varepsilon}}+\left\Vert
\partial_{t}\eta_{L}\left(  t\right)  \right\Vert _{L^{\infty}}\\
&  \leq2C_{1}\left\Vert \left(  \eta_{0},u_{0}\right)  \right\Vert
_{E_{b,d,r}^{s,\varepsilon}}+\left\Vert \partial_{x}u_{L}\left(  t\right)
\right\Vert _{L^{\infty}}+\tilde{C}C_{2}\left\Vert \left(  \partial_{x}%
\eta_{0},\partial_{x}u_{0}\right)  \right\Vert _{X_{b,d,r}^{s-1,\varepsilon}%
}\\
&  \leq C_{3}\left\Vert \left(  \eta_{0},u_{0}\right)  \right\Vert
_{E_{b,d,r}^{s,\varepsilon}},
\end{align*}
where we used $\left(  \text{\ref{conc2}}\right)  $, $\left(
\text{\ref{normaliniara}}\right)  $\ and $\left(  \text{\ref{conc1}}\right)  $.

Proving the uniqueness of the solution is done in the following lines. Let us
suppose that $\left(  \bar{\eta}^{1},\bar{u}^{1}\right)  $, $\left(  \bar
{\eta}^{2},\bar{u}^{2}\right)  \in$ $\mathcal{C}\left(  \left[  0,T\right]
,E_{b,d,r}^{s}\right)  $ are two solutions of $\left(  \text{\ref{BBM1d}%
}\right)  $. Then, we observe that:%
\begin{equation}
\left\{
\begin{array}
[c]{l}%
\left(  I-\varepsilon b\partial_{xx}^{2}\right)  \partial_{t}\delta\bar{\eta
}+\partial_{x}\delta\bar{u}+\varepsilon\partial_{x}\left(  \bar{u}^{1}%
\delta\bar{\eta}+\bar{\eta}^{2}\delta\bar{u}\right)  =0,\\
\left(  I-\varepsilon d\partial_{xx}^{2}\right)  \partial_{t}\delta\bar
{u}+\partial_{x}\delta\bar{\eta}+\varepsilon\bar{u}^{1}\partial_{x}\delta
\bar{u}+\varepsilon\delta\bar{u}\partial_{x}\bar{u}^{2}=0,\\
\delta\bar{\eta}_{|t=0}=0,\text{ }\delta\bar{u}_{|t=0}=0.
\end{array}
\right.  \label{delta1}%
\end{equation}
Writing the first equation of $\left(  \text{\ref{delta1}}\right)  $ in
integral form we get that
\[
\delta\bar{\eta}\left(  t\right)  =\int_{0}^{t}\left(  I-\varepsilon
b\partial_{xx}^{2}\right)  ^{-1}\left(  \partial_{x}\delta\bar{u}+\delta
\bar{\eta}\partial_{x}\bar{u}^{1}+\bar{u}^{1}\partial_{x}\delta\bar{\eta}%
+\bar{\eta}^{2}\partial_{x}\delta\bar{u}+\delta\bar{u}\partial_{x}\bar{\eta
}^{2}\right)  d\tau
\]
and because%
\[
\partial_{x}\delta\bar{u}+\delta\bar{\eta}\partial_{x}\bar{u}^{1}+\bar{u}%
^{1}\partial_{x}\delta\bar{\eta}+\bar{\eta}^{2}\partial_{x}\delta\bar
{u}+\delta\bar{u}\partial_{x}\bar{\eta}^{2}\in L^{2}\left(  \mathbb{R}\right)
\]
we get that
\[
\delta\bar{\eta}\in H^{2\operatorname*{sgn}\left(  b\right)  }\left(
\mathbb{R}\right)  .
\]
Similarly%
\[
\delta\bar{u}\in H^{2\operatorname*{sgn}\left(  d\right)  }\left(
\mathbb{R}\right)  .
\]
Thus, multiplying the first equation of $\left(  \text{\ref{delta1}}\right)  $
with $\delta\bar{\eta}$ $\ $and the second one with $\delta\bar{u}$ and using
Gronwall's lemma will lead us to the conclusion $\left(  \delta\bar{\eta
},\delta\bar{u}\right)  =\left(  0,0\right)  $ on $\left[  0,T\right]  $. This
ends the proof of Theorem \ref{Teorema2}.

\subsection{The 2d case: proof of Theorem \ref{Teorema3}}

For the reader's convenience, let us rewrite below the problem that we wish to solve:%

\begin{equation}
\left\{
\begin{array}
[c]{l}%
\left(  I-\varepsilon b\Delta\right)  \partial_{t}\bar{\eta}%
+\operatorname{div}\bar{V}+\varepsilon\operatorname{div}\left(  \bar{\eta}%
\bar{V}\right)  =0,\\
\left(  I-\varepsilon d\Delta\right)  \partial_{t}\bar{V}+\nabla\bar{\eta
}+\varepsilon\bar{V}\cdot\nabla\bar{V}=0,\\
\bar{\eta}_{|t=0}\left(  x,y\right)  =\eta_{0}\left(  x\right)  +\phi\left(
x,y\right)  ,\\
\bar{V}_{|t=0}\left(  x,y\right)  =\left(  u_{0}\left(  x\right)  ,0\right)
+\psi\left(  x,y\right)  .
\end{array}
\right.  \label{BBM2d}%
\end{equation}

According to Theorem \ref{Teorema2}, there exist $\varepsilon_{0},C^{1D}>0$,
depending on $\sigma,b,d$ and the norm of the initial data $\left(  \eta
_{0},u_{0}\right)  $ such that for any $\varepsilon\leq\varepsilon_{0}$ we may
consider $\left(  \eta^{\varepsilon,1D},u^{\varepsilon,1D}\right)
\in\mathcal{C}\left(  [0,\frac{C^{1D}}{\varepsilon}],E_{b,d,r}^{\sigma}\left(
\mathbb{R}\right)  \right)  $ the unique solution of the $1$-dimensional
system $\left(  \text{\ref{BBM1d}}\right)  $ with initial data $\left(
\eta_{0},u_{0}\right)  $ and
\begin{equation}
\sup_{t\in\lbrack0,\frac{C^{1D}}{\varepsilon}]}\left\Vert \left(
\eta^{\varepsilon,1D}\left(  t\right)  ,u^{\varepsilon,1D}\left(  t\right)
\right)  \right\Vert _{E_{b,d,r}^{\sigma,\varepsilon}}+\sup_{t\in
\lbrack0,\frac{C^{1D}}{\varepsilon}]}\left\Vert \partial_{t}\eta
^{\varepsilon,1D}\left(  t\right)  \right\Vert _{L^{\infty}}\leq\tilde{C}%
_{1}\left\Vert \left(  \eta_{0},u_{0}\right)  \right\Vert _{E_{b,d,r}%
^{\sigma,\varepsilon}},\label{ineg21}%
\end{equation}
with $\tilde{C}_{1}$ some numerical constant. Let us observe that if we
consider
\[
\left\{
\begin{array}
[c]{c}%
\eta_{1}^{\varepsilon}\left(  t,x,y\right)  =\eta^{1D}\left(  t,x\right)  ,\\
V_{1}^{\varepsilon}\left(  t,x,y\right)  =\left(  u^{1D}\left(  t,x\right)
,0\right)  ,
\end{array}
\right.
\]
for all $\left(  t,x,y\right)  \in\lbrack0,\frac{C^{1D}}{\varepsilon}%
]\times\mathbb{R}^{2}$, then:
\[
\left\{
\begin{array}
[c]{l}%
\left(  I-\varepsilon b\Delta\right)  \partial_{t}\eta_{1}^{\varepsilon
}+\operatorname{div}V_{1}^{\varepsilon}+\varepsilon\operatorname{div}\left(
\eta_{1}^{\varepsilon}V_{1}^{\varepsilon}\right)  =0,\\
\left(  I-\varepsilon d\Delta\right)  \partial_{t}V_{1}^{\varepsilon}%
+\nabla\eta_{1}^{\varepsilon}+\varepsilon V_{1}^{\varepsilon}\cdot\nabla
V_{1}^{\varepsilon}=0,\\
\eta_{1|t=0}^{\varepsilon}\left(  x,y\right)  =\eta_{0}\left(  x\right)
,\text{ }V_{1|t=0}^{\varepsilon}=(u_{0}\left(  x\right)  ,0).
\end{array}
\right.
\]
Observe that we have $\partial_{x}\eta^{\varepsilon,1D}\left(  t\right)  \in
B_{2,r}^{\sigma-1}\left(  \mathbb{R}\right)  \hookrightarrow B_{\infty,\infty
}^{\sigma-\frac{3}{2}}\left(  \mathbb{R}\right)  $. Let us observe that for
all $\left(  x_{1},y_{1}\right)  $,$\left(  x_{2},y_{2}\right)  \in
\mathbb{R}^{2}$ such that
\[
\left\vert x_{1}-x_{2}\right\vert ^{2}+\left\vert y_{1}-y_{2}\right\vert
^{2}\leq1
\]
the following holds true%
\begin{align*}
\left\vert \nabla\eta_{1}^{\varepsilon}\left(  t,x_{1},y_{1}\right)
-\nabla\eta_{1}^{\varepsilon}\left(  t,x_{2},y_{2}\right)  \right\vert  &
=\left\vert \partial_{x}\eta^{\varepsilon,1D}\left(  t,x_{1}\right)
-\partial_{x}\eta^{\varepsilon,1D}\left(  t,x_{2}\right)  \right\vert \\
&  \leq C\left\Vert \partial_{x}\eta^{\varepsilon,1D}\left(  t\right)
\right\Vert _{B_{\infty,\infty}^{\sigma-\frac{3}{2}}\left(  \mathbb{R}\right)
}\left\vert x_{1}-x_{2}\right\vert ^{\left(  \sigma-\frac{3}{2}\right)
-\left[  \sigma-\frac{3}{2}\right]  }.
\end{align*}
Using the fact that $B_{\infty,\infty}^{\sigma-\frac{3}{2}}\left(
\mathbb{R}^{2}\right)  =\mathcal{C}^{\sigma-\frac{3}{2},\sigma-\frac{3}%
{2}-\left[  \sigma-\frac{3}{2}\right]  }\left(  \mathbb{R}^{2}\right)  $, see
Remark \ref{observatieH} and relation $\left(  \text{\ref{sigma}}\right)  $,
we get that $\nabla\eta_{1}^{\varepsilon}\left(  t\right)  =\left(
\partial_{x}\eta^{\varepsilon,1D}\left(  t\right)  ,0\right)  \in B_{\infty
,r}^{s}\left(  \mathbb{R}^{2}\right)  $ with
\begin{equation}
\left\Vert \nabla\eta_{1}^{\varepsilon}\left(  t\right)  \right\Vert
_{B_{\infty,r}^{s}\left(  \mathbb{R}^{2}\right)  }\leq C\left\Vert
\partial_{x}\eta^{\varepsilon,1D}\left(  t\right)  \right\Vert _{B_{2,r}%
^{\sigma-1}\left(  \mathbb{R}\right)  }.\label{ineg2.2}%
\end{equation}
Similarly, we get:%
\begin{equation}
\left\Vert \nabla V_{1}^{\varepsilon}\right\Vert _{B_{\infty,r}^{s}\left(
\mathbb{R}^{2}\right)  }\leq C\left\Vert \partial_{x}u^{\varepsilon
,1D}\right\Vert _{B_{2,r}^{\sigma-1}\left(  \mathbb{R}\right)  }%
.\label{ineg2.3}%
\end{equation}
It is also clear that:
\begin{equation}
\left\Vert \left(  \eta_{1}^{\varepsilon},V_{1}^{\varepsilon}\right)
\right\Vert _{L^{\infty}\left(  \mathbb{R}^{2}\right)  }=\left\Vert \left(
\eta^{\varepsilon,1D},u^{\varepsilon,1D}\right)  \right\Vert _{L^{\infty
}\left(  \mathbb{R}\right)  }\text{ and }\left\Vert \partial_{t}\eta
_{1}^{\varepsilon}\right\Vert _{L^{\infty}\left(  \mathbb{R}^{2}\right)
}=\left\Vert \partial_{t}\eta^{\varepsilon,1D}\right\Vert _{L^{\infty}\left(
\mathbb{R}\right)  },\label{ineg2.4}%
\end{equation}
and thus, using $\left(  \text{\ref{ineg2.2}}\right)  $, $\left(
\text{\ref{ineg2.3}}\right)  $, $\left(  \text{\ref{ineg2.4}}\right)  $
respectively the uniform bounds of $\left(  \text{\ref{ineg21}}\right)  $, we
get that:%
\begin{gather}
\sup_{\varepsilon\in\lbrack0,\varepsilon_{0}]}\sup_{t\in\lbrack0,\frac{C_{1D}%
}{\varepsilon}]}\left(  \left\Vert \left(  \eta_{1}^{\varepsilon}\left(
t\right)  ,V_{1}^{\varepsilon}\left(  t\right)  \right)  \right\Vert
_{L^{\infty}\left(  \mathbb{R}^{2}\right)  }+\left\Vert \left(  \nabla\eta
_{1}^{\varepsilon}\left(  t\right)  ,\nabla V_{1}^{\varepsilon}\left(
t\right)  \right)  \right\Vert _{B_{\infty,r}^{s}\left(  \mathbb{R}%
^{2}\right)  }\right)  +\sup_{t\in\lbrack0,\frac{C_{1D}}{\varepsilon}%
]}\left\Vert \partial_{t}\eta_{1}^{\varepsilon}\left(  t\right)  \right\Vert
_{L^{\infty}\left(  \mathbb{R}^{2}\right)  }\nonumber\\
\leq\tilde{C}_{2}\left\Vert \left(  \eta_{0},u_{0}\right)  \right\Vert
_{E_{b,d,r}^{s,1}}.\label{uniformT2}%
\end{gather}
Let us consider the system:%
\[
\left\{
\begin{array}
[c]{l}%
\left(  I-\varepsilon b\Delta\right)  \partial_{t}\eta+\operatorname{div}%
V+\varepsilon\operatorname{div}\left(  \eta V_{1}^{\varepsilon}+\eta
_{1}^{\varepsilon}V+\eta V\right)  =0,\\
\left(  I-\varepsilon d\Delta\right)  \partial_{t}V+\nabla\eta+\varepsilon
\left(  V_{1}^{\varepsilon}+V\right)  \cdot\nabla V+\varepsilon V\cdot\nabla
V_{1}^{\varepsilon}=0,\\
\eta_{|t=0}=\phi,\text{ }V_{|t=0}=\psi.
\end{array}
\right.
\]
Using the uniform bounds of $\left(  \eta_{1}^{\varepsilon},V_{1}%
^{\varepsilon}\right)  _{\varepsilon\leq\varepsilon_{0}}$ and owing to Theorem
\ref{Teorema1.1}, there exist two real numbers $C^{2D}\leq C^{1D}$,
$\varepsilon_{1}\leq\varepsilon_{0}$ depending on $s,\sigma,b,d$ and on
$\left\Vert \left(  \eta_{0},u_{0}\right)  \right\Vert _{E_{b,d,r}^{s,1}%
}+\left\Vert \left(  \phi,\psi\right)  \right\Vert _{X_{b,d,r}^{s,1}}$ such
that for any $\varepsilon\leq\varepsilon_{1}$ we can uniquely construct
$\left(  \eta^{\varepsilon,2D},V^{\varepsilon,2D}\right)  \in\mathcal{C}%
\left(  [0,\frac{C^{2D}}{\varepsilon}],X_{b,d,r}^{s}\right)  $ and, moreover,
\begin{equation}
\sup_{t\in\left[  0,\frac{C^{2D}}{\varepsilon}\right]  }\left\Vert \left(
\eta^{\varepsilon,2D}\left(  t\right)  ,V^{\varepsilon,2D}\left(  t\right)
\right)  \right\Vert _{X_{b,d,r}^{s,\varepsilon}}\leq\tilde{C}_{3}\left(
\left\Vert \left(  \eta_{0},u_{0}\right)  \right\Vert _{E_{b,d,r}%
^{s,\varepsilon}}+\left\Vert \left(  \phi,\psi\right)  \right\Vert
_{X_{b,d,r}^{s,\varepsilon}}\right)  \label{ineg22}%
\end{equation}
where $\tilde{C}_{3}$ is numerical constant. Then, we see that defining%
\[
\left\{
\begin{array}
[c]{l}%
\bar{\eta}^{\varepsilon}\left(  t,x,y\right)  =\eta^{\varepsilon,1D}\left(
t,x\right)  +\eta^{\varepsilon,2D}\left(  t,x,y\right)  ,\\
\bar{V}^{\varepsilon}\left(  t,x,y\right)  =\left(  u^{\varepsilon,1D}\left(
t,x\right)  ,0\right)  +V^{\varepsilon,2D}\left(  t,x,y\right)
\end{array}
\right.
\]
for all $\left(  t,x,y\right)  \in\lbrack0,\frac{C^{2D}}{\varepsilon}%
]\times\mathbb{R}^{2}$ is a solution of $\left(  \text{\ref{weaklydispersive1}%
}\right)  $ which, by construction, belongs to $M_{b,d,r}^{\sigma,s}$.
Moreover, owing to $\left(  \text{\ref{ineg21}}\right)  $ and $\left(
\text{\ref{ineg22}}\right)  $ there exists a constant $\tilde{C}_{4}$ such
that:
\[
\left\Vert \left(  \bar{\eta}^{\varepsilon},\bar{V}^{\varepsilon}\right)
\right\Vert _{M_{b,d,r}^{\sigma,s,\varepsilon}}\leq\tilde{C}_{4}\left(
\left\Vert \left(  \eta_{0},u_{0}\right)  \right\Vert _{E_{b,d,r}%
^{s,\varepsilon}}+\left\Vert \left(  \phi,\psi\right)  \right\Vert
_{X_{b,d,r}^{s,\varepsilon}}\right)  .
\]
We proceed by proving the uniqueness property. Consider two solutions
$\left(  \eta^{i}+\phi^{i},\left(  u^{i}+\psi_{1}^{i},\psi_{2}^{i}\right)
\right)  \in\mathcal{C}\left(  \left[  0,T\right]  ;M_{b,d,r}^{\sigma
,s}\right)  $ with the same initial data and we introduce the following
notations:%
\[
\left\{
\begin{array}
[c]{c}%
\delta\eta=\eta^{1}-\eta^{2},\delta u=u^{1}-u^{2},\\
\delta\phi=\phi^{1}-\phi^{2},\delta\psi=\psi^{1}-\psi^{2}.
\end{array}
\right.
\]
For any pair $\rho\in\mathcal{C}\left(  \left[  0,T\right]  ;\mathcal{S}%
\left(  \mathbb{R}^{2}\right)  \right)  $ we get that:%
\begin{gather}
\int_{0}^{t}\int_{\mathbb{R}^{2}}\left(  \delta\eta+\delta\phi\right)  \left(
I-\varepsilon b\Delta\right)  \partial_{t}\rho+\int_{0}^{t}\int_{\mathbb{R}%
^{2}}\left(  \varepsilon\left(  \delta\eta+\delta\phi\right)  \left(
u^{1}+\psi_{1}^{1}\right)  +\left(  1+\varepsilon\left(  \eta^{2}+\phi
^{2}\right)  \right)  \left(  \delta u+\delta\psi_{1}\right)  \right)
\partial_{x}\rho\nonumber\\
+\int_{0}^{t}\int_{\mathbb{R}^{2}}\left(  \varepsilon\left(  \delta\eta
+\delta\phi\right)  \psi_{2}^{1}+\left(  1+\varepsilon\left(  \eta^{2}%
+\phi^{2}\right)  \right)  \delta\psi_{2}\right)  \partial_{y}\rho
-\int_{\mathbb{R}^{2}}\left(  \delta\eta\left(  t\right)  +\delta\phi\left(
t\right)  \right)  \left(  I-\varepsilon b\Delta\right)  \rho\left(  t\right)
=0.\label{limita}%
\end{gather}
Taking for any $m\in\mathbb{N}^{\star}$
\[
\rho^{m}\left(  t,x,y\right)  =\frac{1}{m}\tilde{\rho}\left(  t,x\right)
\chi\left(  \frac{y}{m}\right)
\]
with $\tilde{\rho}\in\mathcal{C}\left(  \left[  0,T\right]  ;\mathcal{S}%
\left(  \mathbb{R}\right)  \right)  $ and $\chi\in$ $\mathcal{S}\left(
\mathbb{R}\right)  $ we find that%
\begin{gather*}
\int_{0}^{t}\int_{\mathbb{R}}\delta\eta\left(  I-\varepsilon b\partial
_{xx}^{2}\right)  \partial_{t}\tilde{\rho}+\int_{0}^{t}\int_{\mathbb{R}%
}\left(  \varepsilon\delta\eta u^{1}+\left(  1+\varepsilon\eta^{2}\right)
\delta u\right)  \partial_{x}\tilde{\rho}\\
-\int_{\mathbb{R}}\delta\eta\left(  t\right)  \left(  I-\varepsilon
b\partial_{xx}^{2}\right)  \tilde{\rho}\left(  t\right)  =o\left(  m\right)
\text{ when }m\rightarrow\infty.
\end{gather*}
Proceeding similarly we get that for any $\mu\in\mathcal{C}\left(  \left[
0,T\right]  ;\mathcal{S}\left(  \mathbb{R}\right)  \right)  $%
\begin{gather*}
\int_{0}^{t}\int_{\mathbb{R}}\delta u\left(  I-\varepsilon d\partial_{xx}%
^{2}\right)  \partial_{t}\mu+\int_{0}^{t}\int_{\mathbb{R}}\left(  \delta
\eta+\varepsilon u^{1}\delta u\right)  \partial_{x}\mu+\int_{0}^{t}%
\int_{\mathbb{R}}\varepsilon\delta u\left(  \partial_{x}u^{1}-\partial
_{x}u^{2}\right)  \mu\\
-\int_{\mathbb{R}}\delta u\left(  t\right)  \left(  I-\varepsilon
d\partial_{xx}^{2}\right)  \mu\left(  t\right)  =o\left(  m\right)  \text{
when }m\rightarrow\infty.
\end{gather*}
We thus recover that the pair $\left(  \delta\eta,\delta u\right)  \in$
$\mathcal{C}\left(  \left[  0,T\right]  ;E_{b,d,r}^{\sigma}\left(
\mathbb{R}\right)  \right)  $ is a solution of
\begin{equation}
\left\{
\begin{array}
[c]{l}%
\left(  I-\varepsilon b\partial_{xx}\right)  \partial_{t}\delta\eta
+\partial_{x}\delta u+\varepsilon\partial_{x}\left(  \delta\eta u^{1}+\eta
^{2}\delta u\right)  =0,\\
\left(  I-\varepsilon d\partial_{xx}\right)  \partial_{t}\delta u+\partial
_{x}\eta+\varepsilon u^{1}\partial_{x}\delta u+\varepsilon\delta u\partial
_{x}u^{2}=0,\\
\delta\bar{\eta}_{|t=0}=0,\text{ }\delta\bar{u}_{|t=0}=0.
\end{array}
\right.  \label{diff}%
\end{equation}
According to Theorem \ref{Teorema2} this implies that $\left(  \delta
\eta,\delta u\right)  =\left(  0,0\right)  $. It follows that
\begin{equation}
\left\{
\begin{array}
[c]{l}%
\left(  I-\varepsilon b\Delta\right)  \partial_{t}\delta\phi
+\operatorname{div}\delta\psi+\varepsilon\operatorname{div}\left(  \delta
\phi\psi^{1}+\phi^{2}\delta\psi\right)  =0,\\
\left(  I-\varepsilon d\Delta\right)  \partial_{t}\delta\psi+\nabla\delta
\phi+\varepsilon\psi^{1}\cdot\nabla\delta\psi+\varepsilon\delta\psi\cdot
\nabla\psi^{2}=0,\\
\delta\phi_{|t=0}=0,\text{ }\delta\psi_{|t=0}=0,
\end{array}
\right.  \label{diff2}%
\end{equation}
which by Theorem \ref{Teorema1} implies that $\left(  \delta\phi,\text{
}\delta\psi\right)  =\left(  0,0\right)  $. Thus we obtain that the two
solutions coincide.

\begin{remark}
In the case when $b>0$ and $d>0$ a stronger result holds in the sense that one
may prove stability estimates in $L^{\infty}\left(  \mathbb{R}^{2}\right)
\times\left(  L^{\infty}\left(  \mathbb{R}^{2}\right)  \right)  ^{2}$. This is
done by observing that for $l=1,2$ we have:
\[
\mathcal{F}^{-1}\left(  \frac{\xi_{l}}{1+\left\vert \xi\right\vert ^{2}%
}\right)  \in L^{1}\left(  \mathbb{R}^{2}\right)
\]
see for instance \cite{Bona3}, page $606$.
\end{remark}

\section{Appendix: Littlewood-Paley theory}

We present here a few results of Fourier analysis used through the text. The
full proofs along with other complementary results can be found in
\cite{Dan1}. In the following if $\Omega\subset\mathbb{R}^{n}$ is a domain
then $\mathcal{D}\left(  \Omega\right)  $ will denote the set of smooth
functions on $\Omega$ with compact support and $\mathcal{S}$ will denote the
Schwartz class of functions defined on $\mathbb{R}^{n}$. Also, we consider
$\mathcal{S}^{\prime}$ the set of tempered distributions on $\mathbb{R}^{n}$.

\subsection{The dyadic partition of unity}

Let us begin by recalling the so called Bernstein lemma:

\begin{lemma}
\label{Bernstein}Let $\mathcal{C}$ be a given annulus and $B$ a ball of
$\mathbb{R}^{n}$. Let us also consider any nonnegative integer $k$, a couple
$p,q\in\left[  1,\infty\right]  ^{2}$ with $p\leq q$ and any functions $u,v\in
L^{p}$ such that $\mathrm{Supp}(\hat{u})\subset\lambda B$ and $\mathrm{Supp}%
(\hat{v})\subset\lambda\mathcal{C}$. Then, there exists a constant $C$ such
that the following inequalities hold true:%
\begin{equation}
\sup_{\left\vert \alpha\right\vert =k}\left\Vert \partial^{\alpha}u\right\Vert
_{L^{q}}\leq C^{k+1}\lambda^{k+n\left(  \frac{1}{p}-\frac{1}{q}\right)
}\left\Vert u\right\Vert _{L^{p}}, \label{Bernstein1}%
\end{equation}
respectively
\begin{equation}
C^{-k-1}\lambda^{k}\left\Vert v\right\Vert _{L^{p}}\leq\sup_{\left\vert
\alpha\right\vert =k}\left\Vert \partial^{\alpha}v\right\Vert _{L^{p}}\leq
C^{k+1}\lambda^{k}\left\Vert v\right\Vert _{L^{p}}. \label{Bernstein2}%
\end{equation}

\end{lemma}

Next, let us introduce the dyadic partition of the space:

\begin{proposition}
\label{diadic}Let $\mathcal{C}$ be the annulus $\{\xi\in\mathbb{R}^{n}%
:3/4\leq\left\vert \xi\right\vert \leq8/3\}$. There exist two radial functions
$\chi\in\mathcal{D}(B(0,4/3))$ and $\varphi\in\mathcal{D(C)}$ valued in the
interval $\left[  0,1\right]  $ and such that:%
\begin{align}
\forall\xi &  \in\mathbb{R}^{n}\text{, \ }\chi(\xi)+\sum_{j\geq0}%
\varphi(2^{-j}\xi)=1\text{,}\label{25}\\
\forall\xi &  \in\mathbb{R}^{n}\backslash\{0\}\text{, \ }\sum_{j\in\mathbb{Z}%
}\varphi(2^{-j}\xi)=1\text{,}\label{26}\\
2  &  \leq\left\vert j-j^{\prime}\right\vert \Rightarrow\mathrm{Supp}%
(\varphi(2^{-j}\cdot))\cap\mathrm{Supp}(\varphi(2^{-j^{\prime}}\cdot
))=\emptyset\label{27}\\
j  &  \geq1\Rightarrow\mathrm{Supp}(\chi)\cap\mathrm{Supp}(\varphi(2^{-j}%
\cdot))=\emptyset\label{28}%
\end{align}
the set $\mathcal{\tilde{C}=B(}0,2/3\mathcal{)+C}$ is an annulus and we have%
\begin{equation}
\left\vert j-j^{\prime}\right\vert \geq5\Rightarrow2^{j}\mathcal{C}%
\cap2^{j^{\prime}}\mathcal{\tilde{C}}=\emptyset\text{.} \label{29}%
\end{equation}
Also the following inequalities hold true:%
\begin{align}
\forall\xi &  \in\mathbb{R}^{n}\text{, \ }\frac{1}{2}\leq\chi^{2}(\xi
)+\sum_{j\geq0}\varphi^{2}(2^{-j}\xi)\leq1\text{,}\label{210}\\
\forall\xi &  \in\mathbb{R}^{n}\backslash\{0\}\text{, \ }\frac{1}{2}\leq
\sum_{j\in\mathbb{Z}}\varphi^{2}(2^{-j}\xi)\leq1\text{.} \label{211}%
\end{align}

\end{proposition}

From now on we fix two functions $\chi$ and $\varphi$ satisfying the
assertions of Proposition \ref{diadic} and let us denote by $h$ respectively
$\tilde{h}$ their Fourier inverses. The following lemma shows how to
reconstruct a tempered distribution once we know its nonhomogeneous dyadic
blocks (see $\left(  \text{\ref{diadicNeomogen}}\right)  $).

\begin{lemma}
For any $u\in\mathcal{S}^{\prime}$ we have:%
\[
u=\sum_{j\in\mathbb{Z}}\Delta_{j}u\text{ \ in \ }\mathcal{S}^{\prime}\left(
\mathbb{R}^{n}\right)  .
\]

\end{lemma}

\subsection{Properties of Besov spaces}

The following propositions list important basic properties of Besov spaces
that are used through the paper.

\begin{proposition}
\label{propA}A tempered distribution $u\in S^{\prime}$ belongs to $B_{p,r}%
^{s}$ if and only if there exists a sequence $\left(  c_{j}\right)  _{j}$ such
that $\left(  2^{js}c_{j}\right)  _{j}\in\ell^{r}(\mathbb{Z)}$ with norm $1$
and a constant $C\left(  u\right)  >0$ such that for any $j\in\mathbb{Z}$ we
have%
\[
\left\Vert \Delta_{j}u\right\Vert _{L^{p}}\leq C\left(  u\right)  c_{j}.
\]

\end{proposition}

\begin{proposition}
\label{PropBesov}Let $\ s,\tilde{s}\in\mathbb{R}$ and $r,\tilde{r}\in\left[
1,\infty\right]  $.

\begin{itemize}
\item $B_{p,r}^{s}$ is a Banach space which is continuously embedded in
$\mathcal{S}^{\prime}$.

\item The inclusion $B_{p,r}^{s}\subset B_{p,\tilde{r}}^{\tilde{s}}$ is
continuous whenever $\tilde{s}<s$ or $s=\tilde{s}$ and $\tilde{r}>r\,$.

\item We have the following continuous inclusion $B_{p,1}^{\frac{n}{p}}%
\subset\mathcal{C}_{0}$\footnote{$\mathcal{C}_{0}$ is the space of continuous
bounded functions which decay at infinity.}$(\subset L^{\infty})$.

\item (Fatou property) If $\left(  u_{n}\right)  _{n\in\mathbb{N}}$ is a
bounded sequence of $B_{p,r}^{s}$ which tends to $u$ in $\mathcal{S}^{\prime}$
then $u\in B_{p,r}^{s}$ and%
\[
\left\Vert u\right\Vert _{B_{p,r}^{s}}\leq\liminf_{n}\left\Vert u_{n}%
\right\Vert _{B_{p,r}^{s}}.
\]

\item If $r<\infty$ then
\[
\lim_{j\rightarrow\infty}\left\Vert u-S_{j}u\right\Vert _{B_{p,r}^{s}}=0.
\]

\end{itemize}
\end{proposition}

\begin{proposition}
Let us consider $m\in\mathbb{R}$ and a smooth function $f:\mathbb{R}%
^{n}\rightarrow\mathbb{R}$ such that for all multi-index $\alpha$, there
exists a constant $C_{\alpha}$ such that:%
\[
\forall\xi\in\mathbb{R}^{n}\text{ \ }\left\vert \partial^{\alpha}f\left(
\xi\right)  \right\vert \leq C_{\alpha}\left(  1+\left\vert \xi\right\vert
\right)  ^{m-\left\vert \alpha\right\vert }.
\]
Then the operator $f\left(  D\right)  $ is continuous from $B_{p,r}^{s}$ to
$B_{p,r}^{s-m}$.
\end{proposition}

\begin{proposition}
\label{compact}Let $1\leq r\leq\infty$, $s\in\mathbb{R}$ and $\varepsilon>0$.
For all $\phi\in\mathcal{S}$, the map $u\rightarrow\phi u$ is compact from
$B_{2,r}^{s}$ to $B_{2,r}^{s-\varepsilon}$.
\end{proposition}

The next result deals with product estimates in nonhomogeneous Besov spaces.

\begin{theorem}
\label{produs}A constant $C$ exists such that the following holds true.
Consider a real number $s>0$ and any $\left(  p,p_{1},p_{2},p_{3}%
,p_{4},r\right)  \in\left[  1,\infty\right]  ^{6}$ such that%
\[
\frac{1}{p}=\frac{1}{p_{1}}+\frac{1}{p_{2}}=\frac{1}{p_{3}}+\frac{1}{p_{4}}.
\]
For any $\left(  u,v\right)  \in L^{p_{2}}\times\left(  L^{p_{4}}\cap
B_{p_{1},r}^{s}\right)  $ such that $\nabla u\in B_{p_{3,r}}^{s-1}$ we have:%
\[
\left\Vert uv\right\Vert _{B_{p,r}^{s}}\leq\left\Vert u\right\Vert _{L^{p_{2}%
}}\left\Vert v\right\Vert _{B_{p_{1},r}^{s}}+\left\Vert \nabla u\right\Vert
_{B_{p_{3,r}}^{s-1}}\left\Vert v\right\Vert _{L^{p_{4}}}.
\]

\end{theorem}

All the above statements are given or immediate consequences of the results
presented in \cite{Dan1}, pages $107-108$.

\subsection{Commutator estimates}

This section is devoted to establish commutator-type estimates both in
homogeneous and in nonhomogeneous Besov spaces. We begin by stating the
following basic, yet very useful lemma:

\begin{lemma}
\label{Comutator1}Let us consider $\theta$ a $\mathcal{C}^{1}$ function on
$\mathbb{R}^{n}$ such that $\left(  1+\left\vert \cdot\right\vert \right)
\hat{\theta}\in L^{1}$. Let us also consider $p,q\in\left[  1,\infty\right]  $
such that:%
\[
\frac{1}{r}:=\frac{1}{p}+\frac{1}{q}\leq1.
\]
Then, there exists a constant $C$ such that for any Lipschitz function $a$
with gradient in $L^{p}$, any function $b\in L^{q}$ and any positive $\lambda
$:%
\[
\left\Vert \left[  \theta\left(  \lambda^{-1}D\right)  ,a\right]  b\right\Vert
_{L^{r}}\leq C\lambda^{-1}\left\Vert \nabla a\right\Vert _{L^{p}}\left\Vert
b\right\Vert _{L^{q}}.
\]
In particular, when $\theta=\varphi$ and $\lambda=2^{j}$ we get that:%
\[
\left\Vert \left[  \Delta_{j},a\right]  b\right\Vert _{L^{r}}\leq
C2^{-j}\left\Vert \nabla a\right\Vert _{L^{p}}\left\Vert b\right\Vert _{L^{q}%
}.
\]

\end{lemma}

\begin{proposition}
\label{comutneomogen}Let us consider $s>0$ and $\left(  p,p_{1},r\right)
\in\left[  1,\infty\right]  ^{3}$ such that $p\leq p_{1}$. Also, let us
consider two tempered distributions $\left(  u,v\right)  $ such that $u\in
B_{\infty,\infty}^{M}\left(  \mathbb{R}^{n}\right)  $, $\nabla u\in L^{\infty
}\left(  \mathbb{R}^{n}\right)  \cap B_{p_{1},r}^{s-1}\left(  \mathbb{R}%
^{n}\right)  $ and $\nabla v\in L^{p_{2}}\left(  \mathbb{R}^{n}\right)  \cap
B_{p,r}^{s}\left(  \mathbb{R}^{n}\right)  $ where $M$ is some strictly
negative real number and $\frac{1}{p_{2}}=\frac{1}{p}-\frac{1}{p_{1}}$. We
denote by%
\[
R_{j}=[\Delta_{j},u]\partial^{\alpha}v=\Delta_{j}\left(  u\partial^{\alpha
}v\right)  -u\Delta_{j}\partial^{\alpha}v,
\]
where $\alpha\in\overline{1,n}$. Then, the following estimate holds true:%
\[
\left\Vert \left(  2^{js}\left\Vert R_{j}\right\Vert _{L^{p}}\right)
\right\Vert _{\ell^{r}(\mathbb{Z)}}\lesssim_{s}\left\Vert \nabla u\right\Vert
_{L^{\infty}}\left\Vert \nabla v\right\Vert _{B_{p,r}^{s-1}}+\left\Vert \nabla
v\right\Vert _{L^{p_{2}}}\left\Vert \nabla u\right\Vert _{B_{p_{1},r}^{s-1}}.
\]

\end{proposition}

\begin{proof}
The main ingredients of the proof are Lemma \ref{Comutator1} and the Bony
decomposition of the product into the paraproduct and remainder. For the
definitions and some properties we refer to \cite{Dan1} pages 106-108. We
begin by considering%
\[
\tilde{u}=u-\dot{S}_{-1}u
\]
and we write that%
\[
R_{j}=\sum_{i=1,6}R_{j}^{i}%
\]
where%
\[
\left\{
\begin{array}
[c]{ll}%
R_{j}^{1}=\Delta_{j}\left(  T_{\tilde{u}}\partial^{\alpha}v\right)
-T_{\tilde{u}}\Delta_{j}\partial^{\alpha}v, & R_{j}^{2}=\Delta_{j}\left(
T_{\partial^{\alpha}v}\tilde{u}\right)  ,\\
R_{j}^{3}=T_{\Delta_{j}\partial^{\alpha}v}\tilde{u}, & R_{j}^{4}=\Delta
_{j}R\left(  \tilde{u},\partial^{\alpha}v\right)  ,\\
R_{j}^{5}=R\left(  \tilde{u},\Delta_{j}\partial^{\alpha}v\right)  . &
R_{j}^{6}=\left[  \Delta_{j},S_{-1}u\right]  \partial^{\alpha}v.
\end{array}
\right.
\]
We begin with $R_{j}^{1}$. We notice that $\left(  2^{lM}\left\Vert
S_{l-1}\tilde{u}\Delta_{l}\partial^{\alpha}v\right\Vert _{L^{\infty}}\right)
_{l\geq-1}\in\ell^{\infty}$ and
\[
\left\Vert \left(  2^{lM}\left\Vert S_{l-1}\tilde{u}\Delta_{l}\partial
^{\alpha}v\right\Vert _{L^{\infty}}\right)  _{l\geq-1}\right\Vert
_{\ell^{\infty}}\leq C\left\Vert u\right\Vert _{B_{\infty,\infty}^{M}%
}\left\Vert \nabla v\right\Vert _{L^{\infty}}%
\]
such that the series $\sum_{l\geq-1}S_{l-1}u\Delta_{l}\partial^{\alpha}v$ is
convergent. Thus, owing to Lemma \ref{Comutator1} we may write that%
\[
2^{js}\left\Vert R_{j}^{1}\right\Vert _{L^{p}}\leq\sum_{\left\vert
j-l\right\vert \leq4}2^{\left(  j-l\right)  s}2^{ls}\left\Vert \left[
\Delta_{j},S_{l-1}\tilde{u}\right]  \Delta_{l}\partial^{\alpha}v\right\Vert
_{L^{p}}\leq C\left\Vert \nabla u\right\Vert _{L^{\infty}}\sum_{\left\vert
j-l\right\vert \leq4}2^{\left(  j-l\right)  (s-1)}2^{l(s-1)}\left\Vert
\Delta_{l}\nabla v\right\Vert _{L^{p}}.
\]
From Young's inequality we conclude that%
\[
\left\Vert \left(  2^{js}\left\Vert R_{j}^{1}\right\Vert _{L^{p}}\right)
\right\Vert _{\ell^{r}(\mathbb{Z)}}\lesssim_{s}\left\Vert \nabla u\right\Vert
_{L^{\infty}}\left\Vert \nabla v\right\Vert _{B_{p,r}^{s-1}}.
\]
Let us pass to $R_{j}^{2}$. We have that:%
\begin{align*}
2^{js}\left\Vert R_{j}^{2}\right\Vert _{L^{p}}  &  \leq\sum_{\left\vert
j-l\right\vert \leq4}2^{\left(  j-l\right)  s}2^{ls}\left\Vert S_{l-1}%
\partial^{\alpha}v\Delta_{l}\tilde{u}\right\Vert _{L^{p}}\leq C\left\Vert
\nabla v\right\Vert _{L^{p_{2}}}\sum_{\left\vert j-l\right\vert \leq
4}2^{\left(  j-l\right)  s}2^{ls}\left\Vert \Delta_{l}\tilde{u}\right\Vert
_{L^{p_{1}}}\\
&  \leq C\left\Vert \nabla v\right\Vert _{L^{p_{2}}}\sum_{\left\vert
j-l\right\vert \leq4}2^{\left(  j-l\right)  s}2^{l(s-1)}\left\Vert \Delta
_{l}\nabla\tilde{u}\right\Vert _{L^{p_{1}}}.
\end{align*}
From Young's inequality we conclude that%
\[
\left\Vert \left(  2^{js}\left\Vert R_{j}^{2}\right\Vert _{L^{p}}\right)
\right\Vert _{\ell^{r}(\mathbb{Z)}}\lesssim_{s}\left\Vert \nabla v\right\Vert
_{L^{p_{2}}}\left\Vert \nabla u\right\Vert _{B_{p_{1},r}^{s-1}}.
\]
The third term is treated similar:%
\[
2^{js}\left\Vert R_{j}^{3}\right\Vert _{L^{p}}\leq\sum_{j-l\leq5}2^{\left(
j-l\right)  s}2^{ls}\left\Vert S_{l-1}\partial^{\alpha}v\Delta_{l}\tilde
{u}\right\Vert _{L^{p}}\leq C\left\Vert \nabla v\right\Vert _{L^{p_{2}}}%
\sum_{j-l\leq5}2^{\left(  j-l\right)  s}2^{l(s-1)}\left\Vert \Delta_{l}%
\nabla\tilde{u}\right\Vert _{L^{p_{1}}}.
\]
Because $s>0$ we can apply Young's inequality and obtain that%
\[
\left\Vert \left(  2^{js}\left\Vert R_{j}^{3}\right\Vert _{L^{p}}\right)
\right\Vert _{\ell^{r}(\mathbb{Z)}}\lesssim_{s}\left\Vert \nabla v\right\Vert
_{L^{p_{2}}}\left\Vert \nabla u\right\Vert _{B_{p_{1},r}^{s-1}}.
\]
As for the fourth term%
\[
2^{js}\left\Vert R_{j}^{4}\right\Vert _{L^{p}}\leq\sum_{j-l\leq5}2^{\left(
j-l\right)  s}2^{ls}\left\Vert \Delta_{l}\tilde{u}\right\Vert _{L^{\infty}%
}\left\Vert \tilde{\Delta}_{l}\partial^{\alpha}v\right\Vert _{L^{p}}\leq
C\sum_{j-l\leq5}2^{\left(  j-l\right)  s}2^{l(s-1)}\left\Vert \Delta_{l}%
\nabla\tilde{u}\right\Vert _{L^{\infty}}\left\Vert \tilde{\Delta}_{l}%
\partial^{\alpha}v\right\Vert _{L^{p}}.
\]
Because $s>0$ we can apply Young's inequality and obtain that:%
\[
\left\Vert \left(  2^{js}\left\Vert R_{j}^{4}\right\Vert _{L^{p}}\right)
\right\Vert _{\ell^{r}(\mathbb{Z)}}\lesssim_{s}\left\Vert \nabla u\right\Vert
_{L^{\infty}}\left\Vert \nabla v\right\Vert _{B_{p,r}^{s-1}}.
\]
In a similar manner we get that:%
\[
\left\Vert \left(  2^{js}\left\Vert R_{j}^{5}\right\Vert _{L^{p}}\right)
\right\Vert _{\ell^{r}(\mathbb{Z)}}\lesssim_{s}\left\Vert \nabla u\right\Vert
_{L^{\infty}}\left\Vert \nabla v\right\Vert _{B_{p,r}^{s-1}}.
\]
Finally, the sixth term is treated as it follows. First we write that there
exists a integer $N_{0}$ such that:%
\[
\left[  \Delta_{j},\dot{S}_{-1}u\right]  \partial^{\alpha}v=\sum_{l\geq
-1}\left[  \Delta_{j},\dot{S}_{-1}u\right]  \Delta_{l}\partial^{\alpha}%
v=\sum_{\left\vert l-j\right\vert \leq N_{0}}\left[  \Delta_{j},\dot{S}%
_{-1}u\right]  \Delta_{l}\partial^{\alpha}v.
\]
This comes from:%
\[
\mathrm{Supp}\left(  \dot{\Delta}_{j}\left(  \dot{S}_{-1}u\Delta_{l}%
\partial^{\alpha}v\right)  \right)  \subset2^{j}\mathcal{C\cap}\left(
B\left(  0,\frac{2}{3}\right)  +2^{l}\mathcal{C}\right)  \subset
2^{j}\mathcal{C\cap}2^{l}\left(  B\left(  0,\frac{2}{3}\right)  +\mathcal{C}%
\right)
\]
and from the fact that $B\left(  0,\frac{2}{3}\right)  +\mathcal{C}$ is an
annulus. Thus from Lemma \ref{Comutator1} we get that :%
\[
2^{js}\left\Vert R_{j}^{6}\right\Vert _{L^{p}}\leq\sum2^{\left(  j-l\right)
\left(  s-1\right)  }2^{l(s-1)}\left\Vert \nabla\dot{S}_{-1}u\right\Vert
_{L^{\infty}}\left\Vert \Delta_{l}\partial^{\alpha}v\right\Vert _{L^{p}}%
\]
and consequently%
\[
\left\Vert \left(  2^{js}\left\Vert R_{j}^{6}\right\Vert _{L^{p}}\right)
\right\Vert _{\ell^{r}(\mathbb{Z)}}\lesssim_{s}\left\Vert \nabla u\right\Vert
_{L^{\infty}}\left\Vert \nabla v\right\Vert _{B_{p,r}^{s-1}}.
\]
By putting together the estimates for $R_{j}^{1}$,$R_{j}^{2}$,$R_{j}^{3}%
$,$R_{j}^{4}$,$R_{j}^{5}$,$R_{j}^{6}$, we end the proof of Proposition
\ref{comutneomogen}.
\end{proof}

\end{document}